\documentclass{amsart}

\usepackage{graphicx}
\newtheorem{theorem}{Theorem}[section]
\newtheorem{lemma}[theorem]{Lemma}

\theoremstyle{definition}
\newtheorem{definition}[theorem]{Definition}

\newtheorem{corollary}[theorem]{Corollary}

\theoremstyle{remark}
\newtheorem{remark}[theorem]{Remark}

\numberwithin{equation}{section}

\newcommand{\supp}[1]{{\rm supp\/}(#1)}

\begin{document}

\title{Asymptotic of Cauchy biorthogonal polynomials}

\author{U. Fidalgo}
\address{Department of Mathematics, Applied Mathematics, and Statistics, Case Western Reserve University, Cleveland, Ohio 43403, U.S.A.}
\email{uxf6@case.edu}
\thanks{The first author was supported in part by the research grant MTM2015-65888-C4-2-P of Ministerio de Econom\'{\i}a, Industria y Competitividad, Spain.}

\author{G. L\'opez Lagomasino}
\address{Department of Mathematics, Universidad Carlos III de Madrid, Avenida de la Universidad, 30
CP-28911, Leganés, Madrid, Spain.}
\email{lago@math.uc3m.es}
\thanks{The second author was supported in part by the research grant MTM2015-65888-C4-2-P of Ministerio de Econom\'{\i}a, Industria y Competitividad, Spain.}

\author{S. Medina Peralta}
\address{Instituto de Matem\'atica y F\'isica. Universidad de Talca.  Camino Lircay S/N, Campus Norte,  Talca, Chile}
\email{smedinaperalta@gmail.com}
\thanks{The third author was supported by Conicyt Fondecyt/Postdoctorado/ Proyecto 3170112.}


\date{\today}



\begin{abstract} We consider sequences of biorthogonal polynomials with respect to a Cauchy type convolution kernel and give the weak and ratio asymptotic of the corresponding sequences of biorthogonal polynomials. The construction is intimately related with a mixed type Hermite-Pad\'e approximation problem whose asymptotic properties is also revealed.
\end{abstract}

\maketitle

\medskip


\medskip


\section{Introduction}

Let $\Delta$ be a bounded subinterval of the real line. By $\mathcal{M}(\Delta)$ we denote the class of all finite positive Borel measures $\sigma$  whose support $\mbox{supp}(\sigma)$ has infinitely many points and  $\Delta$ is the smallest interval containing $\mbox{supp}(\sigma)$. Take $m \geq 2$ intervals $\Delta_j$, $j=1,\ldots,m$. Throughout the paper we will assume that
\begin{equation} \label{nonint}
\Delta_j \cap \Delta_{j+1}=\emptyset, \qquad j=1,\ldots, m-1,
\end{equation}
and  $(\sigma_1, \ldots, \sigma_m)$ is an ordered collection of measures such that $\sigma_j \in \mathcal{M}(\Delta_j), j=1,\ldots,m$.

\medskip

It is easy to see that for each $n \in \mathbb{Z}_+ := \{0,1,2,\ldots\}$ there exists a polynomial $Q_n, \deg Q_n \leq n,$ not identically equal to zero, such that
\begin{equation}\label{conddefanm}
\int_{\Delta_1}\cdots \int_{\Delta_m} x_1^\nu\frac{Q_n(x_m) \mbox{d} \sigma_m(x_m) \cdots \mbox{d} \sigma_1(x_1)}{(x_{m-1}-x_m)\cdots (x_1-x_2)}= 0,\qquad \nu=0,\ldots,n-1.
\end{equation}
Finding  $Q_n$ reduces to solving a system of $n$ homogeneous linear equations on the $n+1$ unknown coefficients of the polynomial which always has a non-trivial solution. It can be shown that any non-trivial solution has degree $n$. This entails that $Q_n$ is uniquely determined except for a constant factor. All the zeros of $Q_n$ are simple and lie in the interior of $\Delta_m$ (which we denote $\stackrel{\circ}{\Delta}_m $ and the interior is taken with respect to the Euclidean topology of $\mathbb{R}$). This and other properties of $Q_n$ will be proved below in Lemma \ref{l2}. In the sequel, we normalize $Q_n$ to be monic.

\medskip

The orthogonality relations may be expressed more compactly as follows. When $m=2$ we consider the usual Cauchy kernel $K(x_1,x_2)= (x_1-x_2)^{-1}$. For $m > 2$ we take
\[ K(x_1,x_m) = \int_{\Delta_2}\int_{\Delta_3}\cdots \int_{\Delta_{m-1} }
		\frac{\, d \sigma_{m-1}(x_{m-1})\cdots \, d\sigma_3(x_3)d\sigma_2(x_2)   }
		{(x_{m-1}-x_m)(x_{m-2}-x_{m-1})\cdots\,(x_2-x_3)(x_1-x_2)  }.
\]
With this notation, \eqref{conddefanm} adopts the form
\begin{equation}\label{conddefanm2}
\int_{\Delta_1\times \Delta_m} x_1^\nu K(x_1,x_m) Q_n(x_m) \,\mbox{d} \sigma_1(x_1) \,\mbox{d} \sigma_m(x_m)  = 0,\qquad \nu=0,\ldots,n-1.
\end{equation}
Taking into account what was said above, there exist two sequences of monic polynomials $(P_n), (Q_n), n \in \mathbb{Z}_+,$ such that for each $n$, $\deg(P_n) = \deg(Q_n) = n$ and
\begin{equation}\label{biort}
\int_{\Delta_1\times \Delta_m} P_k(x_1) K(x_1,x_m) Q_n(x_m) \,\mbox{d} \sigma_1(x_1) \,\mbox{d} \sigma_m(x_m)  = C_n \delta_{k,n},\quad C_n \neq 0.
\end{equation}
As usual, $\delta_{k,n} = 0, k \neq n, \delta_{n,n} = 1$.

\medskip

These two sequences of polynomials are said to be biorthogonal with respect to $(\sigma_1,\ldots,\sigma_m)$. Notice that the ordering of the measures is important in the definition of the kernel and thus in the definition of biorhogonality. In \cite{Bertola:CBOPs} the authors introduce the concept of biorthogonality with respect to a totally positive kernel. The kernels we have introduced do not fall in that category (except when $m=2$) and, therefore, we will follow a different approach.

\medskip

When $m=2$ biorthogonal polynomials appear in the analysis of the two matrix model \cite{BGS} and for finding discrete solutions of the Degasperis-Procesi equation \cite{Bertola:CBOPs} through a Hermite-Pad\'e approximation problem for two discrete measures. Motivated in \cite{Bertola:CBOPs}, the approximation problem was extended  in \cite{LMS} for arbitrary $m \geq 2$ and general measures proving the convergence of the method.

\medskip

In this paper, we study the asymptotic properties of the sequences of biorthogonal polynomials $(P_n), (Q_n), n \in \mathbb{Z}_+,$ depending on the analytic properties of the measures in $(\sigma_1,\ldots,\sigma_m)$. Before stating the corresponding results, we need to introduce some classes of measures and notation.

\medskip

It is said that $\sigma \in \mathcal{M}(\Delta)$ is regular, and we write $\sigma \in \mbox{\bf Reg}$, if
\[ \lim \gamma_n^{1/n} = \frac{1}{\mbox{cap}(\mbox{supp}(\sigma))},
\]
where $\mbox{cap}(\mbox{supp}(\sigma) )$ denotes the logarithmic capacity of $\mbox{supp}(\sigma)$ and $\gamma_n$ is the leading coefficient of the $n$-th orthonormal polynomial with respect to $\sigma$. See \cite[Theorems 3.1.1, 3.2.1]{stto} for different equivalent forms of defining regular measures and its basic properties. In connection with regular measures it is frequently convenient that the support of the measure be regular. A compact set $E$ is said to be regular when  Green's function corresponding to $\mathbb{C} \setminus E$ with singularity at $\infty$ can be extended continuously to $E$.

\medskip

A measure $\sigma \in \mathcal{M}(\Delta)$ is said to verify the Turan condition when $\sigma' >0 $ almost everywhere on $\Delta$  with respect to the Lebesgue measure. In this case, $\supp{\sigma} = \Delta$ and $\sigma \in \mbox{\bf Reg}$.

\subsection{Statement of the main results}

There are two forms of asymptotic results which play an important role in the general theory of orthogonal polynomials and their applications; namely, their weak asymptotic, connected with the asymptotic zero distribution of their zeros, and the ratio asymptotic (see, for example, \cite{gora}, \cite{kn:Nev1}, \cite{kn:Rak1}, \cite{kn:Rak2}, \cite{kn:Rak3}, \cite{ST}, and \cite{stto}). This interest extends to multiple orthogonal polynomials, which are related with biorthogonal polynomials (see \cite{AGR}, \cite{FLLS}, and \cite{GRS}). The two results we state below follow this line of research.

\medskip

Given a polynomial $Q, \deg(Q) = n$, we denote the associated normalized zero counting measure by
 \[ \mu_Q = \frac{1}{n} \sum_{Q(x) = 0} \delta_x,
 \]
 where $\delta_x$ denotes the Dirac measure with mass $1$ at the point $x$. Our first result says the folowing.

 \begin{theorem} \label{teo1} For each $k=1,\ldots,m$, assume that $\sigma_k \in \mbox{\bf Reg}$ and $\supp{\sigma_k}$ is regular. Then, there exist probability measures $\lambda_1, \lambda_m, \mbox{supp}(\lambda_1) \subset \Delta_1,   \mbox{supp}(\lambda_m) \subset \Delta_m, $ such that
 \begin{equation} \label{asin1}
  *\lim_n \mu_{P_n} = \lambda_1, \qquad *\lim_n \mu_{Q_n} = \lambda_m,
 \end{equation}
 where the convergence is in the weak star topology of measures.
 \end{theorem}

 The measures $\lambda_1,\lambda_m$ are the first and last components of the solution of an associated vector equilibrium problem.  This result is a consequence of Theorem \ref{teo4} (see also Corollary \ref{weakPs}) of Section 3. In Section 3, we will specify the vector equilibrium problem we must deal with and obtain the weak asymptotic of other polynomials and forms associated with this problem.

 \begin{theorem} \label{teo2} Assume that $\sigma_k' > 0$ a.e. on $\Delta_k, k=1,\ldots,m$. Then, there exist analytic functions  $\varphi_1 \in \mathcal{H}(\mathbb{C} \setminus \Delta_1), \varphi_m \in \mathcal{H}(\mathbb{C} \setminus \Delta_m)$ such that
 \begin{equation} \label{asin1}
  \lim_n \frac{P_{n+1}(z)}{P_n(z)} = \varphi_1(z), \qquad \lim_n  \frac{Q_{n+1}(z)}{Q_n(z)} = \varphi_m(z),
 \end{equation}
 uniformly on each compact subset of $\mathbb{C} \setminus \Delta_1$ and $\mathbb{C} \setminus \Delta_m$, respectively.
 \end{theorem}

The functions $\varphi_1$ and $\varphi_m$ are expressed in terms of the branches of a conformal mapping defined on an $m+1$ sheeted Riemann surface of genus zero.  This result is a consequence of the more general Theorem \ref{teo5} (see also Corollary \ref{pesratio}) stated and proved in Section 4 where the Riemann surface is built, the functions $\varphi_1,\varphi_m$ are identified and the ratio asymptotic of other polynomials and forms related with this problem are given.

\medskip

The contents of Sections 3-4 have been described previously. In Section 2 we establish a series of auxiliary results needed in the proofs of the main results. In particular, the existence of the sequences of biorthogonal polynomials is established as well as some properties of their zeros.

\section{Multi-orthogonality relations}

 The results of this section have an algebraic flavor but are indispensable in all what follows. Some may be extracted from \cite{LMS} but we will include the proofs when it is essential to make the reading more comprehensive.

 \subsection{Nikishin system}  Nikishin systems were first introduced in \cite{nik}.
Let $\Delta_{\alpha}, \Delta_{\beta}$ be two bounded intervals contained in the real such that $\Delta_{\alpha}\cap\Delta_{\beta}=\emptyset$. Take $\sigma_{\alpha} \in {\mathcal{M}}(\Delta_{\alpha})$ and $\sigma_{\beta} \in {\mathcal{M}}(\Delta_{\beta})$. Using the  differential notation, we define a third measure $\langle \sigma_{\alpha},\sigma_{\beta} \rangle$ as follows
\[d \langle \sigma_{\alpha},\sigma_{\beta} \rangle (x) := \widehat{\sigma}_{\beta}(x) d\sigma_{\alpha}(x), \qquad \widehat{\sigma}_{\beta}(z) = \int  \frac{d\sigma_{\beta}(x)}{z-x},\]
where $\widehat{\sigma}_{\beta}$ is the Cauchy transform of $\sigma_{\beta}$.


\medskip

Consider a collection  of intervals $\Delta_j, j=1,\ldots,m,$ verifying \eqref{nonint} and measures $\sigma_j \in \mathcal{M}(\Delta_j)$.

\begin{definition} We say that $ (s_{1,1},\ldots,s_{1,m}) = {\mathcal{N}}(\sigma_1,\ldots,\sigma_m)$, where
	\begin{equation} \label{eq:ss}
	s_{1,1} = \sigma_1, \quad s_{1,2} = \langle \sigma_1,\sigma_2 \rangle,  \quad \ldots \quad s_{1,m} = \langle \sigma_1, \sigma_2,\ldots,\sigma_m  \rangle
	\end{equation}
	is the \textit{Nikishin system} of measures generated by $(\sigma_1,\ldots,\sigma_m)$. Here, $s_{1,j}, j \geq 3,$ is defined inductively by taking
\[ \langle \sigma_1, \sigma_2, \ldots,\sigma_j \rangle = \langle \sigma_1, \langle \sigma_2, \ldots,\sigma_j \rangle \rangle.
\]
\end{definition}
%

\medskip

The definition of a Nikishin system can be extended to the case when the intervals $\Delta_j$ are unbounded or touching. The results of this section remain valid when the Nikishin systems  are constructed  following the more general definition given in \cite{LMS}. However, the asymptotic results given in Sections 3-4 require that we use the more restricted version presented here (which, incidentally, coincides with its original formulation in \cite{nik}).

\medskip

In what follows, for $1\leq j, k\leq m$, we denote
\begin{equation} \label{eq:sjk}
s_{j,k} := \langle \sigma_j,\sigma_{j+1},\ldots,\sigma_k \rangle, \quad j <  k, \qquad s_{j,k} := \langle \sigma_j,\sigma_{j-1},\ldots,\sigma_k\rangle, \qquad j > k.
\end{equation}

\medskip

We will make frequent use of \cite[Theorem 1.3]{LS}. For convenience of the reader, we state it here as a lemma. With the present assumptions, the statements are immediate consequences of Cauchy's integral formula, Cauchy's theorem, and the Fubini theorem.

\begin{lemma} \label{reduc} Let $(s_{1,1},\ldots,s_{1,m}) = \mathcal{N}(\sigma_1,\ldots,\sigma_m)$ be given. Assume that there exist polynomials with real coefficients $\ell_0,\ldots,\ell_m$ and a polynomial $w$ with real coefficients whose zeros lie in $\mathbb{C} \setminus \Delta_1$  such that
\[\frac{\mathcal{L}_0(z)}{w(z)} \in \mathcal{H}(\mathbb{C} \setminus \Delta_1)\qquad \mbox{and} \qquad \frac{\mathcal{L}_0(z)}{w(z)} = \mathcal{O}\left(\frac{1}{z^N}\right), \quad z \to \infty,
\]
where $\mathcal{L}_0  := \ell_0 + \sum_{k=1}^m \ell_k  \widehat{s}_{1,k} $ and $N \geq 1$. Let $\mathcal{L}_1  := \ell_1 + \sum_{k=2}^m \ell_k  \widehat{s}_{2,k} $. Then
\begin{equation} \label{eq:3}
\frac{\mathcal{L}_0(z)}{w(z)} = \int \frac{\mathcal{L}_1(x)}{(z-x)} \frac{{\rm d}\sigma_1(x)}{w(x)}.
\end{equation}
If $N \geq 2$, we also have
\begin{equation} \label{eq:4}
\int x^{\nu}  \mathcal{L}_1(x)  \frac{{\rm d}\sigma_1(x)}{w(x)} = 0, \qquad \nu = 0,\ldots, N -2.
\end{equation}
In particular, $\mathcal{L}_1$ has at least $N -1$ sign changes in  $\stackrel{\circ}{\Delta}_1 $.
\end{lemma}

\subsection{Multi-level Hermite-Pad\'e approximation}

We will show shortly that the biorthogonal polynomials $Q_n$ are intimitely connected with a mixed (multilevel) type Hermite-Pad\'e approximation problem introduced in \cite{LMS}. Let us start with the definition.

\begin{definition}{}\label{MTPm2} Consider the Nikishin system $\mathcal{N}(\sigma_1, \sigma_2,\ldots, \sigma_{m})$. Then, for each $n \in \mathbb{N},$ there exist  polynomials $a_{n,0},a_{n,1},\ldots, a_{n,m}$ with $\deg a_{n,j}\leq n-1, j=0,1\ldots,m-1,$  and $\deg a_{n,m}\leq n$, not all identically equal to zero, called \textit{multi-level (ML) Hermite-Pad\'e polynomials} that verify:
\begin{align}
\mathcal{A}_{n,0}(z) := \left(a_{n,0}-a_{n,1}\widehat{s}_{1,1}+a_{n,2}\widehat{s}_{1,2}\cdots+ (-1)^{m}a_{n,m}\widehat{s}_{1,m}\right)(z)=\mathcal{O}(1/z^{n+1})\label{tipoIam}\\
\mathcal{A}_{n,1}(z) :=\left(-a_{n,1}+a_{n,2}\widehat{s}_{2,2}-a_{n,3}\widehat{s}_{2,3}\cdots+ (-1)^{m}a_{n,m}\widehat{s}_{2,m}\right)(z)=\mathcal{O}(1/z)\label{tipoIbm}\\
........................................................................................\nonumber\\
\mathcal{A}_{n,m-1}(z) :=\left((-1)^{m-1}a_{n,m-1}+(-1)^{m}a_{n,m}\widehat{s}_{m,m}\right)(z)=\mathcal{O}(1/z)\label{tipoIdm},
\end{align}
where $\mathcal{O}(\cdot)$ is as $z \to \infty$.
By extension, we take $\mathcal{A}_{n,m} = (-1)^m a_{n,m}$.
\end{definition}

The existence of $a_{n,k}, k=0,\ldots,m,$ is obtained solving a homogeneous linear system of $(n+1)m$ equations on the $(n+1)m +1$ coefficients of the polynomials. Among other properties, in \cite{LMS} (see also Lemma \ref{l2} below) it was shown that $\deg a_{n,m} = n$ and the vector of polynomials $(a_{n,0},\ldots,a_{n,m})$ is uniquely determined up to a constant factor. Consequently, the linear form $\mathcal{A}_{n,0}$ is uniquely determined up to a constant factor  and  we normalize it so that $\mathcal{A}_{n,m} = (-1)^m a_{n,m}$ is monic.

\medskip

From \eqref{eq:3} applied  with $w \equiv 1$ it readily follows that
\[ \mathcal{A}_{n,j}(z) = \int \frac{\mathcal{A}_{n,j+1}(x)}{z-x} {\rm d}\sigma_{j+1}(x), \qquad j=0,\ldots,m-1.
\]
Consequently, for $j=0,\ldots,m-1$
\begin{equation} \label{formint} \mathcal{A}_{n,j}(z) = \int\cdots\int \frac{\mathcal{A}_{n,m}(x_m){\rm d}\sigma_{j+1}(x_{j+1})\cdots {\rm d}\sigma_{m}(x_{m})}{(z - x_{j+1})(x_{j+1} - x_{j+2})\cdots(x_{m-1}-x_m)}.
\end{equation}
When $j=1$ notice that
\begin{equation} \label{an1} \mathcal{A}_{n,1}(x_1) = \int \mathcal{A}_{n,m}(x_m) K(x_1,x_m) {\rm d}\sigma_m(x_m).
\end{equation}


Some of the statements of the next two results may be extracted from \cite{LMS}. However, new notation is introduced and several formulas do not appear explicitly in that paper so, for  convenience of the reader, we include a full proof.

\begin{lemma}\label{l2} Consider the Nikishin system $\mathcal{N}(\sigma_1, \sigma_2,\ldots, \sigma_{m})$. For each fixed $n \in \mathbb{Z}_+$ and $j=1,\ldots,m$, $\mathcal{A}_{n,j}$ has exactly $n$ zeros in $\mathbb{C} \setminus \Delta_{j+1}$ they are all simple and lie in $\stackrel{\circ}{\Delta}_{j}$ $({\Delta}_{m+1} = \emptyset)$. $\mathcal{A}_{n,0}$ has no zero in $\mathbb{C} \setminus \Delta_{1}$. Let $Q_{n,j}$ denote the monic polynomial of degree $n$ whose zeros are those of  $\mathcal{A}_{n,j}$ in $\Delta_j$.  We have $Q_{n,m}  = Q_n$ is the $n$-th biorthogonal polynomial verifying \eqref{conddefanm}. For each $j=0,\ldots,m-1$,
\begin{equation} \label{int1} \frac{\mathcal{A}_{n,j}(z)}{Q_{n,j}(z)} = \int   \frac{\mathcal{A}_{n,j+1}(x)}{z-x}  \frac{ {\rm d}\sigma_{j+1}(x)}{Q_{n,j}(x)},
\end{equation}
where $Q_{n,0} \equiv 1$, and
\begin{equation} \label{int2}
\int x^{\nu}  \mathcal{A}_{n,j+1}(x)  \frac{{\rm d}\sigma_{j+1}(x)}{Q_{n,j}(x)} = 0, \qquad \nu = 0,\ldots, n-1,
\end{equation}
\end{lemma}

\begin{proof} Fix $n \in \mathbb{Z}_+$. According to \eqref{tipoIam}, \eqref{eq:3}, and \eqref{eq:4}
\[ \mathcal{A}_{n,0}(z) = \int  \frac{\mathcal{A}_{n,1}(x)  {{\rm d}\sigma_1(x)}}{z-x}
\]
and
\[
\int x^{\nu}  \mathcal{A}_{n,1}(x)  {{\rm d}\sigma_1(x)} = 0, \qquad \nu = 0,\ldots, n-1.
\]
Therefore, $\mathcal{A}_{n,1}$ has at least $n$ sign changes on $\stackrel{\circ}{\Delta}_{1}$. Should the right hand of \eqref{tipoIam} be $\mathcal{O}(1/z^{n+2})$ or  $\mathcal{A}_{n,0}$ have some zero in $\mathbb{C} \setminus \Delta_{1}$ the use of \eqref{eq:4} would allow us to conclude that the number of sign changes of $\mathcal{A}_{n,1}$ on $\Delta_1$ would be at least $n+1$.

\medskip

Let $Q_{n,1}^*$ be a monic polynomial with real coefficients constructed as follows. It contains as zeros all the points where $\mathcal{A}_{n,1}$ changes sign on $\stackrel{\circ}{\Delta}_{1}$ taking account of their multiplicity (by the identity principle there can be at most a finite number of such points). Should $\mathcal{A}_{n,1}$ have any other root in $\mathbb{C} \setminus \Delta_{2}$ different from the ones taken above, we assign to $Q_{n,1}^*$ one such zero and its complex conjugate if it is a complex number. This is possible because the functions $\mathcal{A}_{n,j}$ are symmetric with respect to the real line and its non real roots come in conjugate pairs.
If $\mathcal{A}_{n,1}$ has exactly $n$ simple zeros on $\stackrel{\circ}{\Delta}_{1}$ and no other root in $\mathbb{C} \setminus \Delta_{2}$ then $Q_{n,1}^*$ is the polynomial denoted  $Q_{n,1}$ in the statement of the lemma; otherwise, $\deg Q_{n,1}^* > n$. We will show that the second option is not possible.

\medskip

By the form in which $Q_{n,1}^*$ was chosen
\[\frac{\mathcal{A}_{n,1}(z)}{Q_{n,1}^*(z)} \in \mathcal{H}(\mathbb{C} \setminus \Delta_2)\qquad \mbox{and} \qquad \frac{\mathcal{A}_{n,1}(z)}{Q_{n,1}^*(z)} = \mathcal{O}\left(\frac{1}{z^{N_1}}\right), \quad z \to \infty,
\]
where $N_1 > n+1$ if either $\deg Q_{n,1}^* > n$ or the expansion in the right hand side of \eqref{tipoIdm} starts at $1/z^2$; otherwise,  $N_1 = n+1, \deg Q_{n,1}^* = n$ and $Q_{n,1}^* = Q_{n,1}$. From \eqref{eq:3} and \eqref{eq:4}
\[\frac{\mathcal{A}_{n,1}(z)}{Q_{n,1}^*(z)} = \int   \frac{\mathcal{A}_{n,2}(x)}{z-x}  \frac{ {\rm d}\sigma_2(x)}{Q_{n,1}^*(x)}
\]
and
\[
\int x^{\nu}  \mathcal{A}_{n,2}(x)  \frac{{\rm d}\sigma_2(x)}{Q_{n,1}^*(x)} = 0, \qquad \nu = 0,\ldots, N_1 -2,
\]
which implies that $\mathcal{A}_{n,2}$ has at least $N_1 -1$ sign changes on $\stackrel{\circ}{\Delta}_{2}$.

\medskip

Now, we can proceed as before defining $Q_{n,2}^*$ similar to the way in which $Q_{n,1}^*$ was chosen. Repeating the arguments employed above, we have
\[\frac{\mathcal{A}_{n,2}(z)}{Q_{n,2}^*(z)} = \int   \frac{\mathcal{A}_{n,3}(x)}{z-x}  \frac{ {\rm d}\sigma_3(x)}{Q_{n,2}^*(x)}
\]
and
\[
\int x^{\nu}  \mathcal{A}_{n,3}(x)  \frac{{\rm d}\sigma_3(x)}{Q_{n,2}^*(x)} = 0, \qquad \nu = 0,\ldots, N_2 -2,
\]
where $N_2 > n+1$ if either $\deg Q_{n,2}^* > n$ or the asymptotic expansion at $\infty$ of $\mathcal{A}_{n,2}$ starts at $1/z^2$. Otherwise, $N_2 = n+1. \deg Q_{n,2}^* = n$ and $Q_{n,2}^* = Q_{n,2}$. In particular,  $\mathcal{A}_{n,3}$ has at least $N_2 -1$ sign changes on $\stackrel{\circ}{\Delta}_{3}$.

\medskip

Following this line of reasoning, for each $j=0,\ldots,m-1$ we can define polynomials $Q_{n,j}^*$ with real coefficients whose zeros lie in $\mathbb{C} \setminus \Delta_{j+1}$, with at least $n$ sign changes on $\Delta_j$ such that
\[ \frac{\mathcal{A}_{n,j}(z)}{Q_{n,j}^*(z)} = \int   \frac{\mathcal{A}_{n,j+1}(x)}{z-x}  \frac{ {\rm d}\sigma_{j+1}(x)}{Q_{n,j}^*(x)},
\]
where $Q_{n,0}^* \equiv 1$, and
\[
\int x^{\nu}  \mathcal{A}_{n,j+1}(x)  \frac{{\rm d}\sigma_{j+1}(x)}{Q_{n,j}^*(x)} = 0, \qquad \nu = 0,\ldots, N_j -2,
\]
where $N_j > n+1$ if either $\deg Q_{n,j}^* > n$ or the asymptotic expansion at $\infty$ of $\mathcal{A}_{n,j}$ starts at $1/z^2$. Otherwise, $N_j = n+1, \deg Q_{n,j}^* = n$ and $Q_{n,j}^* = Q_{n,j}$.

\medskip

The last relation for $j=m-1$ reduces to
\[
\int x^{\nu}  a_{n,m}(x)  \frac{{\rm d}\sigma_{m}(x)}{Q_{n,m-1}^*(x)} = 0, \qquad \nu = 0,\ldots, N_{m-1} -2
\]
(recall that $\mathcal{A}_{n,m} = (-1)^m a_{n,m}$). Since $\deg a_{n,m} \leq n$, if $N_{m-1} > n+1$ the orthogonality relation would imply that $a_{n,m} \equiv 0$ and because of \eqref{formint} $a_{n,j} = 0, j=0,\ldots,m$ which is not the case. Therefore, $N_{m-1} = n+1$. This readily implies that $N_{j} = n+1, j=1,\ldots,m-1$. Consequently, $\deg Q_{n,j}^* = n, j=1,\ldots,m-1$ its zeros are simple and lie on $\stackrel{\circ}{\Delta}_j$ and $Q_{n,j}^* = Q_{n,j}, j=1,\ldots,m-1 $. Now the orthogonality relations imply that $a_{n,m}$ has exactly $n$ simple zeros on $\stackrel{\circ}{\Delta}_m$ and we can take $Q_{n,m} = (-1)^m a_{n,m}$. With this notation, the relations above render \eqref{int1} and \eqref{int2}. From \eqref{int2} with $j=0$ and the expression for $\mathcal{A}_{n,1}$ given in \eqref{an1} it follows that $Q_{n,m}$ is the $n$-th biorthogonal polynomial $Q_n$ defined in \eqref{conddefanm2}. We have completed the proof. \end{proof}

Set
\begin{equation} \label{Hnj}
 \mathcal{H}_{n,j}:=\frac{ Q_{n,j+1}\mathcal{A}_{n,j}}{Q_{n,j}},\qquad j=0,\ldots,m-1.
\end{equation}
As we did before, we take $Q_{n,0} \equiv Q_{n,m+1} \equiv 1$.

\begin{lemma}\label{l4} Consider the Nikishin system $\mathcal{N}(\sigma_1, \sigma_2,\ldots, \sigma_{m})$. For each fixed $n \in \mathbb{Z}_+$ and $j=0,\ldots,m-1$
\begin{equation} \label{int3}
\int x^{\nu}  Q_{n,j+1}(x)  \frac{\mathcal{H}_{n,j+1}(x){\rm d}\sigma_{j+1}(x)}{Q_{n,j}(x)Q_{n,j+2}(x)} = 0, \qquad \nu = 0,\ldots, n-1,
\end{equation}
and
\begin{equation} \label{int4} \mathcal{H}_{n,j}(z) = \int \frac{ Q^2_{n,j+1}(x)}{z-x} \frac{\mathcal{H}_{n,j+1}(x){\rm d}\sigma_{j+1}(x)}{Q_{n,j}(x)Q_{n,j+2}(x)}.
\end{equation}
\end{lemma}

\begin{proof} It is easy to see that \eqref{int3} is the same as \eqref{int2} with the new notation. Since $\deg{Q_{n,j+1} = n}$,  \eqref{int3} implies that
\[\int \frac{Q_{n,j+1}(z) - Q_{n,j+1}(x)}{z-x}  Q_{n,j+1}(x)  \frac{\mathcal{H}_{n,j+1}(x){\rm d}\sigma_{j+1}(x)}{Q_{n,j}(x)Q_{n,j+2}(x)} = 0,\]
In other words,
\[ Q_{n,j+1}(z)\int \frac{Q_{n,j+1}(x)}{z-x}\frac{\mathcal{H}_{n,j+1}(x){\rm d}\sigma_{j+1}(x)}{Q_{n,j}(x)Q_{n,j+2}(x)} = \int \frac{Q_{n,j+1}^2(x)}{z-x}\frac{\mathcal{H}_{n,j+1}(x){\rm d}\sigma_{j+1}(x)}{Q_{n,j}(x)Q_{n,j+2}(x)}.
\]
However, using \eqref{int1}
\[ \int \frac{Q_{n,j+1}(x)}{z-x}\frac{\mathcal{H}_{n,j+1}(x){\rm d}\sigma_{j+1}(x)}{Q_{n,j}(x)Q_{n,j+2}(x)} = \int \frac{\mathcal{A}_{n,j+1}(x)}{z-x}\frac{ {\rm d}\sigma_{j+1}(x)}{Q_{n,j}(x )} = \frac{\mathcal{A}_{n,j}(z)}{Q_{n,j}(z)}
\]
and the left hand of the previous equality reduces to $\mathcal{H}_{n,j}$. Therefore, \eqref{int4} takes place.
\end{proof}

\begin{remark} We wish to underline that the varying measure
\[\frac{\mathcal{H}_{n,j+1}(x){\rm d}\sigma_{j+1}(x)}{Q_{n,j}(x)Q_{n,j+2}(x)}, \qquad j=0,\ldots,m-1
\]
appearing in \eqref{int3} and \eqref{int4} has constant sign on $\Delta_{j+1}$. Indeed, $\sigma_{j+1}$ has constant sign and its support is contained in $\Delta_{j+1}$. This interval does not intersect $\Delta_{j}$ or $\Delta_{j+2}$ which is where the zeros of $Q_{n,j}$ and $Q_{n,j+2}$ lie, respectively. On the other hand $Q_{n,j+1}$ takes away from $\mathcal{A}_{n,j+1}$ all the zeros it had in $\mathbb{C} \setminus \Delta_{j+2}$; in particular, those in $\Delta_{j+1}$. This observation is of importance later on.
\end{remark}

The next lemma implies that for each $j=1,\ldots,m,$ the sequence $\left(Q_{n+1,j}/Q_{n,j}\right)$, $n \in \mathbb{Z}_+,$ is uniformly bounded on each compact subset of $\mathbb{C} \setminus \Delta_j$. This will be very useful in Section 4. The idea of the proof was borrowed from \cite[Theorem 2.1]{AGR} where a similar problem was treated.

\begin{lemma}\label{l3} Consider the Nikishin system $\mathcal{N}(\sigma_1, \sigma_2,\ldots, \sigma_{m})$. For each $n \in \mathbb{Z}_+$ and $j=1,\ldots,m$ the zeros of $Q_{n,j}$ and $Q_{n+1,j}$ interlace.
\end{lemma}

\begin{proof} First of all notice that the statement is equivalent to proving that the zeros of $\mathcal{A}_{n,j}$ and $\mathcal{A}_{n+1,j}$ in $\mathbb{C} \setminus \Delta_{j+1}$  interlace (recall that $\Delta_{m+1} = \emptyset$). Fix $n\in \mathbb{Z}_+$. Let $A,B \in \mathbb{R}$ be such that $|A| + |B| > 0$ and define the linear forms
\[
	\mathcal{D}_{n,j}(z)=A {\mathcal{A}}_{ n,j}(z)+ B {\mathcal{A}}_{ n+1,j}(z),\qquad j=0,\ldots,m.
	\]
Obviously, $\mathcal{D}_{n,0}(x) = \mathcal{O}(1/z^{n+1})$ and $\mathcal{D}_{n,j}(x) = \mathcal{O}(1/z), j=1,\ldots,m-1$.

\medskip

Arguing with the functions $\mathcal{D}_{n,j}$ as we did with the $\mathcal{A}_{n,j}$ in the proof of Lemma \ref{l2} it is easy to deduce that for each $j=1,\ldots,m$ the function $\mathcal{D}_{n,j}$ has at least $n$ sign changes in $\stackrel{\circ}{\Delta}_{j}$ and at most $n+1$ zeros in $\mathbb{C}\setminus\Delta_{j+1}$. Therefore, all the zeros of $\mathcal{D}_{n,j}$ in $\mathbb{C}\setminus\Delta_{j+1}$ are real and simple.

\medskip

From this statement we can draw the conclusion that $\mathcal{A}_{n,j}$ and $\mathcal{A}_{n+1,j}$ cannot have a common zero. Should such a point $y$ exist the function
\[
	\mathcal{D}_{n,j}(x)={\mathcal{A}}_{ n,j}(x)-\frac{{\mathcal{A}}'_{ n,j}(y)}{{\mathcal{A}}'_{ n+1,j}(y)}{\mathcal{A}}_{n+1,j}(x),
	\]
would have a double zero at $y$ which contradicts the statement above.

\medskip	
	
For each fixed $y \in \mathbb{R} \setminus \Delta_{j+1}$, consider the following linear form
\[
	\mathcal{D}_{n,j,y}(x)= \mathcal{A}_{ n+1,j}(y)\mathcal{A}_{ n,j}(x)-\mathcal{A}_{n,j}(y)\mathcal{A}_{n+1,j}(x),
	\]
Since $\mathcal{D}_{n,j,y}(y)=0$, we have $\mathcal{D}_{n,j,y}^{\prime}(y)\ne 0$. Let $y_1 < y_2$ be two consecutive zeros of $\mathcal{A}_{n+1,j}$ in $\mathbb{R} \setminus \Delta_{j+1}$.
The values $\mathcal{A}_{n,j}(y_1), \mathcal{A}_{n+1,j}'(y_1), \mathcal{A}_{n,j}(y_2), \mathcal{A}_{n+1,j}'(y_2)$ all differ from zero because the zeros are simple and there are no common zeros for consecutive $\mathcal{A}_{n,j}$. Therefore,
\[ \mathcal{D}_{n,j,y_1}'(y_1) = -\mathcal{A}_{n,j}(y_1)\mathcal{A}_{n+1,j}'(y_1) \neq 0, \quad \mathcal{D}_{n,j,y_2}'(y_2) = -\mathcal{A}_{n,j}(y_2)\mathcal{A}_{n+1,j}'(y_2) \neq 0.
\]
But the function $\mathcal{D}_{n,j,y}^{\prime}(y)$ preserves the same sign all along the interval $[y_1,y_2]$. Since $\mathcal{A}_{n+1,j}'$ changes its sign in passing from $y_1$ to $y_2$ so must $\mathcal{A}_{n,j}$ and thus $\mathcal{A}_{n,j}$ must have an intermediate zero between $y_1$ and $y_2$. We are done.
\end{proof}

\subsection{The reversed Nikishin system.} Notice that we can also consider the so called reversed Nikishin system $\mathcal{N}(\sigma_m,\ldots,\sigma_1)$ and with it the corresponding associated ML Hermite Pad\'e approximation. More precisely, for each $k\in \mathbb{N}$ there exist polynomials $b_{k,0},\ldots,b_{k,m}$ such that $\deg b_{k,j} \leq k-1, \deg b_{k,m} \leq k$, not all identically equal to zero, such that
	\begin{align} \mathcal{B}_{k,0}(z):=
		\left(b_{k,0}-b_{k,1}\widehat{s}_{m,m}+b_{k,2}\widehat{s}_{m,m-1}\cdots+ (-1)^{m}b_{k,m}\widehat{s}_{m,1}\right)(z)=\mathcal{O}(1/z^{k+1})\label{dtipoIam}\\
	 \mathcal{B}_{k,1}(z):=	\left(-b_{k,1}+b_{k,2}\widehat{s}_{m-1,m-1} \cdots+ (-1)^{m}b_{k,m}\widehat{s}_{m-1,1}\right)(z)=\mathcal{O}(1/z)\label{dtipoIbm}\\
		........................................................................................\nonumber\\
	 \mathcal{B}_{k,m-1}(z):=	\left((-1)^{m-1}b_{k,m-1}+(-1)^{m}b_{k,m}\widehat{s}_{1,1}\right)(z)=\mathcal{O}(1/z)\label{dtipoIdm}.
	\end{align}
Set $\mathcal{B}_{k,m}(z) = (-1)^m b_{k,m}$.

\medskip

Using what has been proved, for each $j=1,\ldots,m$ the form $\mathcal{B}_{k,j}, j=1,\ldots,m,$ has exactly $k$ zeros in $\mathbb{C} \setminus \Delta_{m-j}, \Delta_0 = \emptyset,$ they are all simple and lie in $\stackrel{\circ}{\Delta}_{m-j +1}$. Accordingly, there exist monic polynomials
${P}_{k,j}, j=1,\ldots,m,$ of degree $k$ whose zeros are the roots of $\mathcal{B}_{k,j}$ in $\Delta_{m-j+1}$, respectively. Normalizing $\mathcal{B}_{k,0}$ so that $\mathcal{B}_{k,m}$ is monic, the polynomial $P_{k,m}$ equals $\mathcal{B}_{k,m}$ and it is the $k$-th biorthogonal polynomial $P_k$ verifying \eqref{biort}. This last statement is the contents of \cite[Theorem 1.5]{LMS} but it readily follows from Lemma \ref{l2}. Lemma \ref{l3} implies that for each $j=1,\ldots,m$ the zeros of $P_{k,j}$ and $P_{k+1,j}$ interlace.

\section{Weak asymptotic.}

Following standard techniques,  the weak asymptotic is derived using arguments from potential theory. Therefore, we will briefly summarize what we need.

\subsection{Preliminaries from potential theory.}

Let $E_k,\, k=1\ldots,m,$ be (not necessarily distinct)
compact subsets of the real line and
\[ \mathcal{C} = (c_{j,k}), \qquad
1\leq j,k \leq m,
\] a real, positive definite, symmetric
matrix of order $m$. $\mathcal{C}$ will be called the
interaction matrix. Let
$\mathcal{M}_1(E_k) $ be the subclass of probability measures in
$\mathcal{M}(E_k).$ Set
\[\mathcal{M}_1= \mathcal{M}_1(E_{1})
\times \cdots \times \mathcal{M}_1(E_{m})  \,.
\]

Given a vector measure $\vec{\mu}=(\mu_{1},\ldots,\,\mu_{m}) \in
\mathcal{M}_1$ and $j= 1,\ldots,m,$ we define the combined
potential
\[W^{\vec{\mu}}_j(x) = \sum_{k=1}^{m} c_{j,k}
V^{\mu_k}(x) \,,
\]
where
\[ V^{\mu_k}(x) = \int \log \frac{1}{|x-t|} \,d\mu_k(t)\,,
\]
denotes the standard logarithmic potential of $\mu_k$. We denote
\[ \omega_j^{\vec{\mu}} = \inf \{W_j^{\vec{\mu}}(x): x \in E_j\} \,, \quad
j=1,\ldots,m\,.
\]

In Chapter 5 of \cite{NiSo}  the authors prove (we state the
result in a form convenient for our purpose)

\begin{lemma} \label{niksor} Assume that the compact sets
$E_k,k=1,\ldots,m,$ are regular. Let $\mathcal{C}$ be a real, positive definite, symmetric
matrix of order $m$. If there exists $\vec{\lambda}  =
(\lambda_{1},\ldots,\lambda_{m})\in
\mathcal{M}_1$ such that for each $j=1,\ldots,m$
\[
W_j^{\vec{\lambda}} (x) = \omega_j^{\vec{\lambda}}\,, \qquad x
\in \supp{\lambda_j}\,,
\]
then $\vec{\lambda}$ is unique. Moreover, if $c_{j,k} \geq 0$
when $E_j \cap E_k \neq \emptyset$, then $\vec{\lambda}$ exists.
\end{lemma}

\medskip

For details on how Lemma \ref{niksor} is derived from  \cite[Chapter
5]{NiSo} see  \cite[Section 4]{bel}. The vector measure $\vec{\lambda}$
is called the equilibrium solution for the vector potential problem determined by the
interaction matrix $\mathcal{C}$ on the system of compact sets
$E_j\,, j = 1,\ldots,m$ and $\omega^{\vec{\lambda}}:= (\omega_1^{\vec{\lambda}},\ldots,\omega_m^{\vec{\lambda}})$ is the vector equilibrium constant. There are other characterizations of the equilibrium measure and constant but we will not dwell into that because they will not be used and their formulation requires introducing additional notions and notation.

\medskip

In the proof of the asymptotic zero distribution of the polynomials $Q_{n,j}$ we take $E_j = \supp{\sigma_j}$.   The interaction matrix is the typical one for problems involving  Nikishin systems. Namely,
\[
\mathcal{C}_{\mathcal{N}}=\left(\begin{array}{r r r r r r}
 1      & -1/2 &  0 & \cdots & 0 & 0 \\
-1/2      &  1 & -1/2 & \cdots & 0 & 0 \\
 0      & -1/2 &  1 & \cdots & 0 & 0 \\
 \vdots &    &    & \ddots &   & \vdots \\
 0      &  0 &  0 & \cdots & 1 & -1/2 \\
 0      &  0 &  0 & \cdots &-1/2 &  1
\end{array}
\right)_{m\times m}
\]
which is a real, symmetric, positive definite matrix with positive diagonal elements.
All the assumptions of Lemma \ref{niksor} are in place and the existence of a unique vector equilibrium measure on the   system of sets $E_j, j=1,\ldots,m$ is guaranteed.

\medskip

We also need

\begin{lemma} \label{lemextremal}
Let $E \subset \mathbb{R}$ be a regular compact set and $\phi$ a continuous function
on $E$. Then, there exists a unique $\lambda \in
\mathcal{M}_1(E)$ and a constant $w$ such that
\[
V^{\lambda}(z)+\phi(z) \left\{ \begin{array}{l} \leq
w,\quad z
\in \supp{\lambda} \,, \\
\geq w, \quad z \in  E\,.
\end{array} \right.
\]
\end{lemma}

In particular, equality takes place on all $\supp{\lambda}$.
If the compact set $E$ is not regular with respect to the
Dirichlet problem, the second part of the statement is true
except on a set $e$ such that $\mbox{cap}(e) =0.$ Theorem I.1.3 in
\cite{ST} contains a proof of this lemma in this context. When $E$
is regular, it is well known that this inequality except on a set
of capacity zero implies the inequality for all points in the set
(cf. Theorem I.4.8 from \cite{ST}). $\lambda$ is called the
equilibrium measure in the presence
of the external field $\phi$ on $E$ and $w$ is  the equilibrium
constant.

\medskip

One last ingredient in the proof of the asymptotic zero distribution of the polynomials $Q_{n,j}$ is provided by the following  lemma. Different
versions of it appear in \cite{gora},  and
\cite{stto}. In \cite{gora}, it was proved assuming that
$\supp{\sigma}$ is an interval on which $\sigma'
> 0$ a.e.  Theorem 3.3.3 in
\cite{stto}  does not cover the type of
external field we need to consider. As stated here, the proof appears in \cite[Lemma 4.2]{FLLS}.

\begin{lemma}\label{gonchar-rakhmanov}
Assume that $\sigma \in \mbox{\bf Reg}$ and $\supp{\sigma} \subset
\mathbb{R}$  is regular.   Let $\{\phi_n\}, n \in \Lambda \subset
\mathbb{Z}_+,$ be a sequence of positive continuous functions on
$\supp{\sigma}$ such that
\begin{equation} \label{eq:phi}
\lim_{n\in \Lambda}\frac{1}{2n}\log\frac{1}{|\phi_n(x)|}= \phi(x)
> -\infty ,
\end{equation}
uniformly on $\supp{\sigma}$. Let  $\{q_n\}, n \in \Lambda,$ be
a sequence of monic polynomials such that $\deg q_n = n$ and
\[
\int x^k q_n(x)\phi_n(x)d\sigma(x)=0,\qquad k=0,\ldots, n-1.
\]
Then
\begin{equation} \label{eq:18}
*\lim_{n \in \Lambda}\mu_{q_n} = \lambda,
\end{equation}
and
\begin{equation} \label{eq:19}
\lim_{n\in \Lambda}\left(\int |q_n(x)|^2\phi_n(x)
d\sigma(x)\right)^{1/{2n}}= e^{-w},
\end{equation}
where $\lambda$ and $w$ are the  equilibrium measure and
equilibrium constant in the presence of the external field $\phi$
on $\supp{\sigma}$. We also have
\begin{equation} \label{eq:H}
\lim_{n \in \Lambda} \left(\frac{|q_n(z)|}{\|q_n
\phi_n^{1/2}\|_E}\right)^{1/n} = \exp{(w -
V^{\lambda}(z))}, \qquad  \mathcal{K} \subset  \mathbb{C} \setminus \Delta,
 \end{equation}
 where $\|\cdot\|_E$ denotes the uniform norm on $E$ and $\Delta$ is the smallest interval containing $\supp{\sigma}$.
\end{lemma}

\subsection{Weak asymptotic and some consequences.} We are ready for the proof of the asymptotic zero distribution of the polynomials $(Q_{n,j})$.

 \begin{theorem} \label{teo4} Assume that $\sigma_j \in \mbox{\bf Reg}$ and $\supp{\sigma_j} = E_j$ is regular for each $j=1,\ldots,m$. Then,
\begin{equation} \label{weak}
*\lim_n \mu_{Q_{n,j}} =  {\lambda}_{j}, \qquad j=1,\ldots,m.
\end{equation}
where $\vec{\lambda} = (\lambda_1,\ldots,\lambda_m) \in
\mathcal{M}_1$ is the vector equilibrium measure determined by the
matrix $\mathcal{C}_{\mathcal{N}}$ on the system of compact
sets $E_j, j=1,\ldots,m$.  Moreover,
\begin{equation} \label{eq:4*}
\lim_n \left|\int  {Q_{n,j}^2(x)}\frac{\mathcal{H}_{n,j}(x)\,{\rm d}\sigma_j(x)}{Q_{n,j-1}(x)Q_{n,j+1}(x)} \right|^{1/2n} = \exp\left(-\sum_{k=j}^{m}
{\omega_k^{\vec{\lambda}}} \right)\,,
\end{equation}
where  $\omega^{\vec{\lambda}} = (\omega_1^{\vec{\lambda}},\ldots,\omega_m^{\vec{\lambda}})$ is the vector equilibrium constant.
\end{theorem}
\begin{proof} The unit ball in the cone of positive Borel measures
is weak star compact; therefore, it is sufficient to show that
each one of the sequences of measures $\{\mu_{Q_{n,j}}\}$,
$n\in\mathbb{Z}_+$, $j=1,\ldots,m,$ has only one
accumulation point which coincides with the corresponding
component of the vector equilibrium measure $\vec{\lambda}$ determined by the
matrix $\mathcal{C}_{\mathcal{N}}$ on the system of compact
sets $E_j, j=1,\ldots,m$.

\medskip

Let
$\Lambda\subset \mathbb{Z}_+$ be  such
that for each $j=1,\ldots,m$
\[
*\lim_{n\in\Lambda}\mu_{Q_{n,j}}=\mu_j.
\]
Notice that $\mu_j\in\mathcal{M}_1(E_j)$, $j=1,\ldots,m$. Taking into account that all the zeros of $Q_{n,j}$ lie in $\Delta_j$, it follows that
\begin{equation}\label{conv-Qnj}
\lim_{n\in\Lambda}|Q_{n,j}(z)|^{1/n}=\exp(- V^{\mu_j}(z)),
\end{equation}
uniformly on compact subsets of $\mathbb{C} \setminus\Delta_j$.

\medskip

Because of the normalization adopted, $(-1)^m a_{n,m} = \mathcal{A}_{n,m} = Q_{n,m}$; consequently, when
$j=m-1$, \eqref{int3} takes the form
\[ \int x^{\nu} Q_{n,m}(x) \frac{{\rm d} \sigma_{m}(x)}{|Q_{n,m-1}(x)|} = 0\,, \qquad \nu=0,\ldots,n-1.
\]
According to (\ref{conv-Qnj})
\[ \lim_{n \in \Lambda}\frac{1}{2n}\log|Q_{n,m-1}(x)| =
-\frac{1}{2 } V^{\mu_{m-1}}(x)\,,
\]
uniformly on $\Delta_{m}$. Using Lemma \ref{gonchar-rakhmanov},
it follows that $\mu_{m}$ is the unique solution of the extremal
problem
\begin{equation} \label{eq:1}
V^{\mu_{m}}(x) - \frac{1}{2}V^{\mu_{m-1}}(x)
\left\{
\begin{array}{l} = \omega_{m},\quad x
\in \supp {\mu_{m}} \,, \\
\geq \omega_{m}, \quad x \in E_{m} \,,
\end{array} \right.
\end{equation}
and
\begin{equation} \label{eq:2}
\lim_{n \in \Lambda} \left(\int \frac{Q_{n,m}^2(x)}{|Q_{n,m-1}(x)|}{\rm d}\sigma_{m}(x)\right)^{1/2n} =
e^{-\omega_{m}}\,.
\end{equation}

Using induction on decreasing values of $j$, let us show that for all
$j = 1,\ldots,m$
\begin{equation} \label{eq:3*}
V^{\mu_j}(x) - \frac{1}{2 }V^{\mu_{j-1}}(x) -
\frac{1}{2 }V^{\mu_{j+1}}(x) +
 \omega_{j+1} \left\{
\begin{array}{l} = \omega_j,\quad x
\in \supp{\mu_j} \,, \\
\geq \omega_j, \quad x \in E_j \,,
\end{array} \right.
\end{equation}
where $V^{\mu_0} \equiv V^{\mu_{m+1}} \equiv 0, \omega_{m+1} =0$, and
\begin{equation} \label{eq:4**}
\lim_{n \in \Lambda} \left(\int  {Q_{n,j}^2(x)} \frac{|\mathcal{H}_{n,j}(x)| {\rm d}\sigma_j(x)}{|Q_{n,j-1}(x)Q_{n,j+1}(x)|} \right)^{1/2n} = e^{-\omega_j}\,,
\end{equation}
where $Q_{n,0} \equiv Q_{n,m+1} \equiv 1$. For $j = m$ these relations
are non other than (\ref{eq:1})-(\ref{eq:2}) and the initial
induction step is settled. Let us assume that the statement is
true for $j+1 \in \{2,\ldots,m\}$ and let us prove it for
$j$.

\medskip

For $j=1,\ldots, m$, the orthogonality relations  \eqref{int3} can be expressed as
\[ \int x^{\nu} Q_{n,j}(x)  \frac{|\mathcal{H}_{n,j}(x)| {\rm d}\sigma_j(x)}{|Q_{n,j-1}(x)Q_{n,j+1}(x)|} = 0\,, \qquad \nu=0,\ldots,n-1\,,
\]
and using \eqref{int4} it follows that
\[ \int x^{\nu} Q_{n,j}(x) \left( \int \frac{Q_{n,j+1}^2(t)}{|x-t|} \frac{|\mathcal{H}_{n,j+1}(t)|{\rm d}\sigma_{j+1}(t)}{ |Q_{n,j}(t)Q_{n,j+2}(t)| }\right)\frac{{\rm d}\sigma_j(x)}{|Q_{n,j-1}(x)Q_{n,j+1}(x)|} = 0\,,
\]
for $\nu=0,\ldots,n-1\,.$

\medskip

Relation (\ref{conv-Qnj}) implies that
\begin{equation}\label{eq:5}
\lim_{n \in \Lambda} \frac{1}{2n}\log|Q_{n,j-1}(x)Q_{n,j+1}(x)| = -
\frac{1}{2 }V^{\mu_{j-1}}(x) -
\frac{1}{2 }V^{\mu_{j+1}}(x)\,,
\end{equation}
uniformly on $\Delta_j.$ (Since $Q_{n,0}\equiv 1$, when
$j=1$ we only get the second term on the right hand side of
this limit.)

\medskip

Set
\begin{equation} \label{Knj} K_{n,j+1} := \left(\int  {Q_{n,j+1}^2(t)}  \frac{ |{\mathcal{H}}_{n,j+1}(t)|{\rm d}\sigma_{j+1}(t)}{|Q_{n,j}(t)Q_{n,j+2}(t)|}
 \right)^{-1/2}.
\end{equation}
It follows that for $x \in \Delta_j$
\[ \frac{1}{\delta_{j+1}^*K_{n,j+1}^2} \leq \int \frac{Q_{n,j+1}^2(t)}{|x-t|}\frac{ |{\mathcal{H}}_{n,j+1}(t)|{\rm d}\sigma_{j+1}(t)}{|Q_{n,j}(t)Q_{n,j+2}(t)|}
\leq \frac{1}{\delta_{j+1}K_{n,j+1}^2},
\]
where $0 < \delta_{j+1} = \inf\{|x-t|: t \in \Delta_{j+1}, x \in
\Delta_j\} \leq \max\{|x-t|: t \in \Delta_{j+1}, x \in \Delta_j\}
= \delta_{j+1}^* < \infty.$ Taking into consideration these
inequalities, from the induction hypothesis, we obtain that
\begin{equation} \label{eq:6}
\lim_{n \in \Lambda} \left(\int \frac{Q_{n,j+1}^2(t)}{|x-t|}\frac{ |{\mathcal{H}}_{n,j+1}(t)|{\rm d}\sigma_{j+1}(t)}{|Q_{n,j}(t)Q_{n,j+2}(t)|}\right)^{1/2n} = e^{- \omega_{j+1} }.
\end{equation}

\medskip

Taking (\ref{eq:5}) and (\ref{eq:6}) into account, Lemma
\ref{gonchar-rakhmanov} yields that $\mu_j$ is the unique solution
of the extremal problem (\ref{eq:3*}) and
\[  \lim_{n \in \Lambda} \left(\int   \int \frac{Q_{n,j+1}^2(t)}{|x-t|}\frac{ |{\mathcal{H}}_{n,j+1}(t)|{\rm d}\sigma_{j+1}(t)}{|Q_{n,j}(t)Q_{n,j+2}(t)|}
\frac{Q_{n,j}^2(x){\rm d}\sigma_j(x)}{|Q_{n,j-1}(x)Q_{n,j+1}(x)|}\right)^{1/2n} = e^{-\omega_j}.
\]
Using \eqref{int4} the previous formula reduces to \eqref{eq:4}. We have concluded the induction.

\medskip

Now, we can rewrite (\ref{eq:3*}) as
\begin{equation} \label{eq:a1}
 V^{\mu_j}(x) - \frac{1}{2}V^{\mu_{j-1}}(x) -
\frac{1}{2 }V^{\mu_{j+1}}(x)  \left\{
\begin{array}{l} = \omega_j',\quad x
\in \supp{\mu_j} \,, \\
\geq \omega_j', \quad x \in E_j \,,
\end{array} \right.
\end{equation}
for $j=1,\ldots,m$, where
\begin{equation}\label{eq:c}
\omega_j' =  \omega_j -  \omega_{j+1}, \qquad (\omega_{m+1} = 0).
\end{equation}
(Recall that the terms with $V^{\mu_0}$ and $V^{\mu_{m+1}}$ do not appear when $j=0$ and $j=m$, respectively.)
By Lemma \ref{niksor},
$\vec{\lambda} = (\mu_{1},\ldots,\mu_{m})$ is the equilibrium solution for the vector potential problem determined by the
interaction matrix $\mathcal{C}_{\mathcal{N}}$ on the system of compact sets
$E_j\,, j = 1,\ldots,m$ and $\omega^{\vec{\lambda}} = (\omega_{1}',\ldots,\omega_{m}') $  is the corresponding vector equilibrium constant.  This is
for any convergent subsequence; since the equilibrium problem does not depend on the sequence of indices $\Lambda$  and the solution is unique we  obtain the limits in \eqref{weak}.

\medskip

From the uniqueness  of the vector equilibrium constant and  (\ref{eq:4**}), we have
\[
\lim_{n \to \infty} \left(\int  {Q_{n,j}^2(x)} \frac{|\mathcal{H}_{n,j}(x)|\, {\rm d}\sigma_j(x)}{|Q_{n,j-1}(x)Q_{n,j+1}(x)|} \right)^{1/2n} = e^{-\omega_j}\,,
\]
On the other hand, from (\ref{eq:c}) it follows that
$\omega_{m} = \omega_{m}^{\vec{\lambda}}$ when
$j=m.$  Suppose that $ \omega_{j+1} = \sum_{k=j+1}^{m}
 {\omega_{k}^{\vec{\lambda}}}$ where $j+1 \in
\{2,\ldots,m\}$. Then, according to (\ref{eq:c})
\[ \omega_j = {\omega_j^{\vec{\lambda}}}  +
 \omega_{j+1} = \sum_{k=j}^{m}
 {\omega_{k}^{\vec{\lambda}}}
\]
and (\ref{eq:4*}) immediately follows.
\end{proof}

\begin{theorem}\label{general}
Assume that $(\sigma_1,\ldots,\sigma_m) \in \mbox{\bf Reg}$ and $\supp{\sigma_j} = E_j, j=1,\ldots,m$ is regular.
Then, for each $j=0,\ldots,m$
\begin{equation}\label{conver1}
\lim_{n \to \infty} |\mathcal{A}_{n,j}(z)|^{1/n}=A_j(z), \qquad {\mathcal{K}} \subset \mathbb{C}\setminus
(\Delta_j \cup \Delta_{j+1})
\end{equation}
$( \Delta_{m +1} = \Delta_0 = \emptyset )$, where
\[
A_j(z)=\exp \left(  V^{ {\lambda}_{j+1}}(z)-
V^{ {\lambda}_j}(z) - 2\sum_{k=j+1}^{m }
{\omega_k^{\vec{\lambda}}} \right),\quad
j=1,\ldots,m-1,
\]
and
\[ A_0(z)=\exp \left(  V^{ {\lambda}_{1}}(z) - 2\sum_{k=1}^{m }
{\omega_k^{\vec{\lambda}}} \right),
\qquad A_{m}(z)=\exp\left(- V^{ {\lambda}_{m}}(z)\right).
\]
$\vec{\lambda} =
( {\lambda}_{1},\ldots, {\lambda}_{m})$ is  the
vector equilibrium measure   and
$(\omega_{1}^{\vec{\lambda}},\ldots,\omega_{m}^{\vec{\lambda}})$
is the vector equilibrium constant for the vector potential
problem determined by the interaction matrix $\mathcal{C}_{\mathcal{N}}$ acting on the system of compact sets $E_j, j=1,\ldots,m$.
\end{theorem}

\begin{proof} If $j=m$ then $ \mathcal{A}_{n,m} = Q_{n,m}$ and
\eqref{weak} directly implies that
\[ \lim_{n\in\Lambda} |\mathcal{A}_{n,m}(z)|^{1/n}
=  \exp\left(- V^{ {\lambda}_{m}}(z)\right), \qquad
\mathcal{K} \subset \mathbb{C} \setminus \Delta_{m}.
\]
For $j \in \{0,\ldots,m-1\}$, from \eqref{int4} we have
\begin{equation} \label{eq:7} \mathcal{A}_{n,j}(z) = \frac{Q_{n,j}(z)}{Q_{n,j+1}(z)}\int \frac{ Q^2_{n,j+1}(x)}{z-x} \frac{\mathcal{H}_{n,j+1}(x){\rm d}\sigma_{j+1}(x)}{Q_{n,j}(x)Q_{n,j+2}(x)},
\end{equation}
where $Q_{n,0}\equiv Q_{n,m+1} \equiv 1$. Now, \eqref{weak} implies
\[ \lim_{n\to \infty}\left|\frac{Q_{n,j}(z)}{Q_{n,j+1}(z)}\right|^{ {1}/{n}}=\exp\left(
V^{{\lambda}_{j+1}}(z)- V^{ {\lambda}_j}(z)\right),
\qquad \mathcal{K} \subset \mathbb{C} \setminus (\Delta_j \cup
\Delta_{j+1})
\]
(we also use that the zeros of $Q_{n,j}$ and $Q_{n,j+1}$ lie in $\Delta_j$ and $\Delta_{j+1}$, respectively). It
remains to find the $n$-th root asymptotic behavior of
the integral.

Fix a compact set $\mathcal{K} \subset \mathbb{C} \setminus
\Delta_{j+1}.$ It is easy to verify that (for the definition of
$K_{n,j+1}$ see \eqref{Knj})
\begin{equation} \label{cotainf} \frac{C_1}{K_{n,j+1}^2} \leq \left|  \int \frac{ Q^2_{n,j+1}(x)}{z-x} \frac{\mathcal{H}_{n,j+1}(x){\rm d}\sigma_{j+1}(x)}{Q_{n,j}(x)Q_{n,j+2}(x)}\right| \leq \frac{C_2}{K_{{\bf
n},j+1}^2} \,,
\end{equation}
where
\[ C_1 = \frac{\min \{ \max\{|u-x|,|v|: z = u+iv\}: z \in \mathcal{K}, x
\in \Delta_{j+1}\}}{ \max\{|z-x|^2: z \in \mathcal{K}, x \in
\Delta_{j+1}\}} > 0
\]
and
\[ C_2 = \frac{1}{\min\{|z-x|: z \in \mathcal{K}, x \in
\Delta_{j+1}\}} < \infty.
\]
Taking into account (\ref{eq:4*})
\begin{equation}\label{eq:9}
\lim_{n \to \infty} \left| \int \frac{ Q^2_{n,j+1}(x)}{z-x} \frac{\mathcal{H}_{n,j+1}(x){\rm d}\sigma_{j+1}(x)}{Q_{n,j}(x)Q_{n,j+2}(x)}\right|^{1/n} = \exp\left(-2
\sum_{k=j+1}^{m} {\omega_k^{\vec{\lambda}}}\right)\,.
\end{equation}
From (\ref{eq:7})-(\ref{eq:9}), we obtain (\ref{conver1}) and we
are done.
\end{proof}

In \cite[Theorem 1.6]{LMS} it was proved that for each $j= 0,\ldots,m-1$
\begin{equation} \label{conver2} \lim_{n\to \infty} \frac{a_{n,j}(z)}{a_{n,m}(z)} = \widehat{s}_{m,j+1}(z),
\end{equation}
uniformly on compact subsets of $\mathbb{C}\setminus \Delta_m$.

\medskip

\begin{corollary} \label{cor1} Assume $\sigma_j  \in \mbox{\bf Reg}$ and $\supp{\sigma_j}$ is regular for each $j=1,\ldots,m$. Then,
\begin{equation}\label{conver3}
\lim_{n \to \infty} |{a}_{n,j}(z)|^{1/n}=A_m(z), \qquad j=0,\ldots,m,
\end{equation}
uniformly on compact subsets of $\mathbb{C}\setminus \Delta_m$.
\end{corollary}
\begin{proof} We have $\mathcal{A}_{n,m} = (-1)^m a_{n,m}$; consequently, \eqref{conver3} when $j=m$ is a consequence of \eqref{conver1}. On the other hand, the function $\widehat{s}_{m,j+1}$ never equals zero in $\mathbb{C}\setminus \Delta_m$; therefore, for the remaining  values of $j$ formula \eqref{conver3} is an immediate consequence of \eqref{conver3} for $j=m$ and \eqref{conver2}.
\end{proof}


Now we wish to use the results obtained to produce estimates of the rate of convergence in \eqref{conver2}. For this we need some properties that we summarize in the next corollary.

\begin{corollary}\label{Medina1}
	Under the assumptions of Theorem \ref{general}, for each $k=1,\ldots, m$ and $j=0, \ldots k-1$ we have
 \begin{equation}\label{Serg2}
	\limsup_{n \to \infty} \left| \frac{ \mathcal{A}_{n,j}}{ \mathcal{A}_{n,k}} \right|^{1/n}
	\end{equation}
	\[
	\leq \exp\left(-V^{ {\lambda}_{k+1}}(z)+V^{ {\lambda}_{k}}(z)+V^{ {\lambda}_{j+1}}(z)-
	V^{ {\lambda}_j}(z) - 2\sum_{\ell=j+1}^{k}
	{\omega_\ell^{\vec{\lambda}}}\right),
	\]
uniformly on compact subsets $\displaystyle \mathcal{K}\subset \mathbb{C}\setminus \left(\Delta_{k}\cup  \Delta_{j+1}\right)$, and
 \begin{equation}\label{Serg}
	\lim_{n \to \infty} \left| \frac{ \mathcal{A}_{n,j}}{ \mathcal{A}_{n,k}} \right|^{1/n}
	\end{equation}
	\[
	=\exp\left(-V^{ {\lambda}_{k+1}}(z)+V^{ {\lambda}_{k}}(z)+V^{ {\lambda}_{j+1}}(z)-
	V^{ {\lambda}_j}(z) - 2\sum_{\ell=j+1}^{k}
	{\omega_\ell^{\vec{\lambda}}}\right),
	\]
uniformly on compact subsets $\mathcal{K}\subset \mathbb{C}\setminus
	\Delta_j \cup \Delta_{j+1} \cup \Delta_{k} \cup \Delta_{k+1}$.
For $k=1,\ldots,m$
 \begin{equation}\label{equilibriumk}
 -V^{ {\lambda}_{k+1}}(z)+2V^{ {\lambda}_{k}}(z)-V^{ {\lambda}_{k-1}}(z) - 2
	{\omega_k^{\vec{\lambda}}}<0,\quad z\in\mathbb{C}\setminus\Delta_{k},
 \end{equation}
by convention, $V^{ {\lambda}_{0}}(z) = V^{ {\lambda}_{m+1}}(z) \equiv 0$. If $k>j+1$	
	\begin{equation}\label{equilibriumjk}
	-V^{ {\lambda}_{k+1}}(z)+V^{ {\lambda}_{k}}(z)+V^{ {\lambda}_{j+1}}(z)-
	V^{ {\lambda}_j}(z) - 2\sum_{\ell=j+1}^{k}
	{\omega_\ell^{\vec{\lambda}}} <0,\quad z\in\mathbb{C},
	\end{equation}
	which implies that the sequence $ \left\lbrace \mathcal{A}_{n,j} /\mathcal{A}_{n,k}\right\rbrace  $  converges to zero with geometric rate on any compact subset of $\mathbb{C}\setminus (\Delta_{k}\cup \Delta_{j+1}) $.
\end{corollary}

\begin{proof} Fix $\displaystyle k\in \left\{1, \ldots, m\right\}$ and $\displaystyle j\in \left\{0,\ldots,k-1\right\}$. From (\ref{eq:7}) we obtain that
\begin{equation}\label{previousciudadprofeta1}
\frac{\mathcal{A}_{n,j}(z)}{\mathcal{A}_{n,k}(z)} = \frac{Q_{n,j}(z)Q_{n,k+1}(z)}{Q_{n,j+1}(z)Q_{n,k}(z)} \frac{\displaystyle \int \frac{ Q^2_{n,j+1}(x)}{z-x} \frac{\mathcal{H}_{n,j+1}(x){\rm d}\sigma_{j+1}(x)}{Q_{n,j}(x)Q_{n,j+2}(x)}}{\displaystyle \int \frac{ Q^2_{n,k+1}(x)}{z-x} \frac{\mathcal{H}_{n,k+1}(x){\rm d}\sigma_{k+1}(x)}{Q_{n,k}(x)Q_{n,k+2}(x)}}
\end{equation}
Using the relation (\ref{weak}) in Theorem \ref{teo4}, it follows that uniformly on each compact subset $\mathcal{K}\subset \mathbb{C}\setminus
	(\Delta_j \cup \Delta_{j+1} \cup \Delta_{k} \cup \Delta_{k+1})$, we have
\[
\lim_{n\to \infty} \left|\frac{(Q_{n,j}Q_{n,k+1})(z)}{(Q_{n,j+1}Q_{n,k})(z)} \right|^{1/n}= \exp \left(-V^{ {\lambda}_{k+1}}+V^{ {\lambda}_{k}}+V^{ {\lambda}_{j+1}} - V^{ {\lambda}_j} \right)(z),
\]
and taking into account (\ref{eq:9}), from (\ref{previousciudadprofeta1}) we obtain (\ref{Serg}).

\medskip

Now, from the principle of descent (see \cite[Appendix III]{stto}), locally uniformly on $\mathbb{C}$ we have
\begin{equation}\label{descenso}
\limsup_{n \to \infty}\left|Q_{n,j}(z)Q_{n,k+1}(z)\right|^{1/n} \leq \exp \left(-V^{\lambda_{k+1}}(z)-V^{\lambda_j}(z)\right).
\end{equation}
Using the lower bound in \eqref{cotainf} (with $j$ replaced by $k$) to estimate the integral in the denominator of \eqref{previousciudadprofeta1}
from below and the previous remarks, \eqref{Serg2} readily follows.

\medskip

According to (\ref{eq:3*}), for $k=1,\ldots,m,$ we have
\begin{equation} \label{igual}
-    V^{ {\lambda}_{k+1}}(z)+2V^{ {\lambda}_{k}}(z)- V^{ {\lambda}_{k-1}}(z) - 2
{\omega_{k}^{\vec{\lambda}}} =0,\quad z
\in \mbox{supp}(\lambda_{k}).
\end{equation}
Recall that all the measures $\lambda_k$ are probability, hence for each $k=2, \ldots, m-1$ the function $- V^{ {\lambda}_{k+1}}(z)+2V^{ {\lambda}_{k}}(z)- V^{ {\lambda}_{k-1}}(z)- 2
{\omega_{k}^{\vec{\lambda}}}$ is harmonic at $z=\infty$, and is subharmonic in $ {\mathbb{C}}\setminus \supp{\lambda_{k}}$.  According to the maximum principle for subharmonic functions we obtain \eqref{equilibriumk}.

When $k=1$, the left hand of \eqref{igual} reduces to  $
- V^{ {\lambda}_{2}}(z)+2V^{ {\lambda}_{1}}(z)- 2 {\omega_{1}^{\vec{\lambda}}}$  which is subharmonic in $\mathbb{C} \setminus \supp{\lambda_1}$ and also subharmonic at $\infty$ since
\[
\lim_{z\to \infty} \left(- V^{ {\lambda}_{2}}(z)+2V^{ {\lambda}_{1}}(z)- 2 {\omega_{1}^{\vec{\lambda}}}\right) =-\infty.
\]
Therefore, we can also use the maximum principle to derive \eqref{equilibriumk}.
The case $k=m$ is completely analogous to the case $k=1$.

When $k>j+1$ we can write
\[
	-V^{ {\lambda}_{k+1}}(z)+V^{ {\lambda}_{k}}(z)+V^{ {\lambda}_{j+1}}(z)-
	V^{ {\lambda}_j}(z) - 2\sum_{\ell=j+1}^{k}
	{\omega_\ell^{\vec{\lambda}}}
\]
\[
= \sum_{\ell=j+1}^k \left( - V^{ {\lambda}_{\ell+1}}(z)+2V^{ {\lambda}_{\ell}}(z)- V^{ {\lambda}_{\ell-1}}(z)- 2 {\omega_{\ell}^{\vec{\lambda}}}\right),
\]
and this sum contains at least two terms because $k > j+1$. Each term is less than or equal to zero in all $\mathbb{C}$ and so is the whole sum. To prove that it is strictly negative it is sufficient to show that at each point there is at least one negative term in the sum.
Let us assume that there is a $z_0 \in \mathbb{C}$ such that
\[ - V^{ {\lambda}_{\ell+1}}(z_0)+2V^{ {\lambda}_{\ell}}(z_0)- V^{ {\lambda}_{\ell-1}}(z_0)- 2 {\omega_{\ell}^{\vec{\lambda}}} = 0, \qquad \ell = j+1,\ldots,k.
\]
By what was proved above, this implies that $z_0 \in \cap_{\ell = j+1}^k \Delta_\ell$. However, this is impossible because consecutive intervals in a Nikishin system are disjoint. From \eqref{Serg2} and \eqref{equilibriumjk} the final statement of the corollary readily follows.
\end{proof}

As a consequence of Corollary \ref{Medina1} we can recover the functions $\widehat{s}_{m-1,j+1},j=0,\ldots,m-2$.

\begin{corollary} \label{cor4} Under the assumptions of Theorem \ref{general}, for each  $j=0, \ldots m-2$, we have
\begin{equation} \label{nivelanterior}
\lim_{n\to \infty} \frac{(a_{n,j} - a_{n,m}\widehat{s}_{m,j+1})(z)}{(a_{n,m-1} - a_{n,m}\widehat{s}_{m,m})(z)} = \widehat{s}_{m-1,j+1}(z),
\end{equation}
uniformly on each compact subset of $\mathbb{C}\setminus  \cup_{\ell=j +1}^{m}\Delta_{\ell}.$
\end{corollary}

\begin{proof} Using formula (2.2) from \cite[Lemma 2.1]{LMS} with $\mathcal{L}_j = \mathcal{A}_{n,j}, j=0,\ldots,m-2$ and $r=m-1$, we obtain the following identity
\begin{equation} \label{formaux}
\mathcal{A}_{n,j}+\sum_{k=j+1}^{m-1} (-1)^{k-j} \widehat{s}_{k,j+1} \mathcal{A}_{n,k}= (-1)^j\left(a_{n,j}-a_{n,m}\widehat{s}_{m,j+1}\right).
\end{equation}
The formula holds at all points where both sides are meaningful. Dividing by $\mathcal{A}_{n,m-1}$, we obtain
\[
\frac{\mathcal{A}_{n,j}}{\mathcal{A}_{n,m-1}}+\sum_{k=j+1}^{m-2} (-1)^{k-j} \widehat{s}_{k,j+1} \frac{\mathcal{A}_{n,k}}{\mathcal{A}_{n,m-1}} + (-1)^{m-1-j} \widehat{s}_{m-1,j+1}= \]
\[(-1)^{m-1 + j}\frac{\left(a_{n,j}-a_{n,m}\widehat{s}_{m,j+1}\right)}{\left(a_{n,m-1}-a_{n,m}\widehat{s}_{m,m}\right)}.
\]
To obtain \eqref{nivelanterior}, it remains to take limit on both sides and make use of the fact that the ratios ${\mathcal{A}_{n,k}}/{\mathcal{A}_{n,m-1}}$ uniformly tend to zero on compact subsets of $\mathbb{C}\setminus  \cup_{\ell=j+1}^{m}\Delta_{\ell}.$
\end{proof}

Incidentally, we wish to mention that \eqref{Serg2} and \eqref{equilibriumjk} imply that \eqref{nivelanterior} takes place with geometric rate.

\medskip

As a consequence of Corollary \ref{Medina1} we can also give explicit expressions for the exact rate of convergence of the limits (\ref{conver2}).

\begin{theorem}\label{velocidad} Assume that $(\sigma_1,\ldots,\sigma_m) \in \mbox{\bf Reg}$ and $\supp{\sigma_j} = E_j, j=1,\ldots,m$ is regular.
Then, for each $j=0,\ldots,m-1$:
\begin{equation}\label{City}
	\lim_{n\to \infty} \left|\left(\frac{a_{n,j}(z)}{a_{n,m}(z)} - \widehat{s}_{m,j+1}(z)\right)h_{n,j}^{-1}(z)\right|^{1/n}=\exp\left(2V^{ {\lambda}_{m}}(z)-
	V^{ {\lambda}_{m-1}}(z) - 2
	{\omega_m^{\vec{\lambda}}}\right)
	\end{equation}
uniformly on each compact subset $\mathcal{K}\subset \mathbb{C} \setminus (\Delta_m \cup \Delta_{m-1})$, where $h_{n,j}$ is a polynomial of degree at most $m-j$ whose roots are the possible zeros which $\frac{a_{n,j}}{a_{n,m}} - \widehat{s}_{m,j+1}$ may have in a neighborhood of $\mathcal{K} \cap \mathbb{R}$.
\end{theorem}

\begin{proof} Let us start out again from \eqref{formaux}, but now we divide it by  $\mathcal{A}_{n,m} = (-1)^m a_{n,m}$. We obtain,
\[
\left|\frac{\mathcal{A}_{n,j}}{\mathcal{A}_{n,m}}+\sum_{k=j+1}^{m-1} (-1)^{k-j} \widehat{s}_{k,j+1} \frac{\mathcal{A}_{n,k}}{\mathcal{A}_{n,m}}\right|= \left|\frac{a_{n,j}}{a_{n,m}}-\widehat{s}_{m,j+1}\right|,
\]
which is equivalent to
\[
\left|\frac{\mathcal{A}_{n,m-1}}{\mathcal{A}_{n,m}}\right| \left|\frac{(-1)^j \mathcal{A}_{n,j}}{\mathcal{A}_{n,m-1}} + \sum_{k=j+1}^{m-1} (-1)^{k} \widehat{s}_{k,j+1} \frac{\mathcal{A}_{n,k}}{\mathcal{A}_{n,m-1}} \right|= \left|\frac{a_{n,j}}{a_{n,m}}-\widehat{s}_{m,j+1}\right|.
\]
Now, $\widehat{s}_{m-1,j+1}(z) \neq 0, z \in \mathbb{C} \setminus \Delta_{m-1}$; consequently, $\lim_{n\to \infty}|\widehat{s}_{m-1,j+1}(z)|^{1/n} = 1$ uniformly on compact subsets of $\mathbb{C} \setminus \Delta_{m-1}$. This, together with \eqref{Serg} and \eqref{equilibriumjk} gives us
\[
\lim_{n\to \infty} \left|\frac{(-1)^j \mathcal{A}_{n,j}}{\mathcal{A}_{n,m-1}} + \sum_{k=j+1}^{m-1} (-1)^{k} \widehat{s}_{k,j+1} \frac{\mathcal{A}_{n,k}}{\mathcal{A}_{n,m-1}} \right|^{1/n}= 1
\]
uniformly on compact subsets of   $\mathbb{C}\setminus  \cup_{\ell=j+1}^{m}\Delta_{\ell}.$ Then,
\[
\lim_{n\to \infty} \left|\frac{a_{n,j}}{a_{n,m}}-\widehat{s}_{m,j+1}\right|^{1/n}  = \lim_{n\to \infty}  \left|\frac{\mathcal{A}_{n,m-1}}{\mathcal{A}_{n,m}}\right|^{1/n},
\]
uniformly on compact subsets of $\mathbb{C}\setminus  \cup_{\ell=j+1}^{m}\Delta_{\ell}.$
Using again \eqref{Serg} on the limit on the right hand side, it follows that
\begin{equation}\label{demedinaalameca}
\lim_{n\to \infty} \left|\frac{a_{n,j}}{a_{n,m}}-\widehat{s}_{m,j+1}\right|^{1/n}=\exp\left(2V^{ {\lambda}_{m}}(z)-
	V^{ {\lambda}_{m-1}}(z) - 2
	{\omega_m^{\vec{\lambda}}}\right)
\end{equation}
uniformly on compact subsets of $ \mathbb{C}\setminus  \cup_{\ell=j+1}^{m}\Delta_{\ell}$. Notice that for such compact sets $h_{n,j}$ can be taken equal to $1$ in \eqref{City} for all $n$.

\medskip

Let us improve \eqref{demedinaalameca} to cover \eqref{City}. For this, it is sufficient to show that for every $x \in \mathbb{R} \setminus (\Delta_m \cup \Delta_{m-1})$ there exists $\varepsilon > 0$ such that \eqref{City} holds true uniformly on the closed disk $\{z:|z-x| \leq \varepsilon\}$. Fix $x \in \mathbb{R} \setminus (\Delta_m \cup \Delta_{m-1})$ and let $d$ be equal to the distance from $x$ to $\Delta_m \cup \Delta_{m-1}$.

\medskip

For each $n\in \mathbb{N}$, the function $\displaystyle \delta_{n,j} := \frac{a_{n,j}}{a_{n,m}}-\widehat{s}_{m,j+1} \in \mathcal{H} \left(\overline{\mathbb{C}}\setminus \Delta_m\right)$ has at least $n-m+j$ sign changes on $\Delta_{m-1}$ (see  \cite[relation (2.26)]{LMS}). It readily follows that $\delta_{n,j}$ can vanish at most $m-j$ times, counting multiplicities, in $\mathbb{C}\setminus \left(\Delta_{m-1}\cup \Delta_m\right)$. Indeed, the opposite implies  that $a_{n,m}$  verifies at least $n+1$ orthogonality relations on $\Delta_m$ with respect to a measure with constant sign which is impossible since $\deg a_{n,m} = n$. Let $r_{n,j}$ denote the monic polynomial  that vanishes at the zeros of $\delta_{n,j}$ in $\{z: |z-x| \leq d/2\}$. Obviously, $\deg r_{n,j} \leq m-j$.

\medskip

The functions $\delta_{n,j}/r_{n,j}$ are analytic and different from zero in $\{z:|z-x| \leq d/2\}$. Therefore, for each $n$ we can define a branch $f_{n,j} = (\delta_{n,j}/r_{n,j})^{1/n}$ holomorphic in $\{z:|z-x| \leq d/2\}$. Let us show that $(f_{n,j}), n \geq 0,$ is uniformly bounded on $\{z:|z-x| \leq d/4\}$. Due to \eqref{conver2} the sequence $|\delta_{n,j}|, n \geq 0,$ is uniformly bounded above on the annulus $\{z: d/4 \leq |z-x| \leq d/2\}$.

\medskip

Let $x_{n,j,\ell}, \ell=1,\ldots,N_n, N_n \leq m-j,$ be the collection of roots of $r_{n,j}$ and $C_{n,j,\ell}$ the circle of radius $d/10(m-j)$ centered at $x_{n,j,\ell}$. The sum of the diameters of these circles does not exceed
\[ (m-j) \frac{d}{5(m-j)} = \frac{d}{5} < \frac{d}{4}.
\]
Since $d/4$ is the width of the annulus $\{z: d/4 \leq |z-x| \leq d/2\}$, there exists a circle $\gamma_n$ of radius $d_n, d/4 \leq d_n \leq d/2,$ centered at $x$ which does not intersect any of the circles $C_{n,j,\ell}, \ell =1,\ldots,N_n$. On $\gamma_n$, we have
\[ \inf_{z\in \gamma_n} |r_{n,j}(z)| \geq \left(\frac{d}{10(m-j)}\right)^{N_n} \geq \min\left\{\frac{d}{10(m-j)},\left(\frac{d}{10(m-j)}\right)^{m-j}\right\}.
\]
By the maximum principle for analytic functions, we have
\[ \sup_{|z-x| \leq d/4} \left|f_{n,j}(z)\right| \leq \sup_{z \in \gamma_n}  \left|f_{n,j}(z)\right| \leq \]
\[\sup_{d/4 \leq |z -x|\leq d/2 } |\delta_{n,j}(z)|^{1/n}/\inf_{z\in \gamma_n} |r_{n,j}(z)|^{1/n} \leq C,
\]
where $C$ is a constant which does not depend on $n$. Consequently, the sequence of functions $(f_{n,j}), n\geq 0,$ is uniformly bounded on each compact subset of $\{z: |z-x| < d/4\}$.
\medskip

Choose any convergent subsequence
\[ \lim_{n \in \Lambda} f_{n,j} = f_{\Lambda}, \qquad z \in D = \{z: |z-x| < d/4\}.
\]
Without loss of generality, taking a sub-subsequence if necessary, we can assume that $\lim_{n\in \Lambda} r_{n,j} = r_j$, where $r_j$ is a polynomial of degree $\leq m-j$. The function $f_\Lambda$ is analytic in the disk $D$. Since $f_{n,j} \neq 0, n \geq 0,$ in $D$ it follows that $f_{\Lambda}$ is either identically equal to zero or never equals zero in $D$. According to \eqref{demedinaalameca}
\begin{equation} \label{analytic}  |f_{\Lambda}(z)| = \exp\left(2V^{ {\lambda}_{m}}(z)-
	V^{ {\lambda}_{m-1}}(z) - 2
	{\omega_m^{\vec{\lambda}}}\right), \qquad z \in D \setminus (\mathbb{R} \cup \{z: r_j(z) = 0\}).
\end{equation}
Now $f_{\Lambda}$ is analytic in $D$ and the right hand side of \eqref{analytic} is the absolute value of an analytic function in $D$ so by analytic continuation \eqref{analytic} holds on all $D$. Consequently,
\[ \lim_{n \to \infty} |f_{n,j}(z)| = \exp\left(2V^{ {\lambda}_{m}}(z)-
	V^{ {\lambda}_{m-1}}(z) - 2
	{\omega_m^{\vec{\lambda}}}\right),
\]
uniformly on compact subsets of $D$.

\medskip

For an arbitrary compact set $\mathcal{K} \subset \mathbb{C} \setminus (\Delta_m \cup \Delta_{m-1})$, in order to construct $h_{n,j}$ we would have to remove all the zeros of $\frac{a_{n,j}}{a_{n,m}} - \widehat{s}_{m,j+1}$ in a neighborhood of $\mathcal{K} \cap \mathbb{R}$. By what was said above the amount of such zeros does not exceed $m-j$. Then, \eqref{City} is obtained taking an appropriate covering of $\mathcal{K}$. For compact subsets away from the real line \eqref{City} follows from \eqref{demedinaalameca} taking $h_{n,j} \equiv 1$.  We are done.
\end{proof}

\begin{remark} We suspect that the zeros of  $\frac{a_{n,j}}{a_{n,m}} - \widehat{s}_{m,j+1}$ lying in $\mathbb{C} \setminus (\Delta_m \cup \Delta_{m-1})$ accumulate on $\Delta_{m} \cup \Delta_{m-1}$ (more specifically on $\Delta_{m-1}$). If this is true, then \eqref{City}  with $h_{n,j} \equiv 1$ holds on compact subsets of $\mathbb{C} \setminus (\Delta_m \cup \Delta_{m-1})$. Consider $\cup_{\ell=j}^m \Delta_\ell$. This set can be written as the union of disjoint intervals. Of those disjoint intervals, let $\widetilde{\Delta}_{m,j}, \widetilde{\Delta}_{m-1,j}$ be the ones containing $\Delta_m$ and $\Delta_{m-1}$, respectively (they may coincide). Using arguments similar to those employed in the proof of Corollary \ref{cor4}, it is not hard to deduce that \eqref{nivelanterior} takes place uniformly on compact subsets of $\mathbb{C} \setminus (\widetilde{\Delta}_{m,j} \cup \widetilde{\Delta}_{m-1,j})$. Since $\widehat{s}_{m-1,j+1}$ and $a_{n,m-1} - a_{n,m}\widehat{s}_{m,m}$ have no zeros or poles in $\mathbb{C} \setminus (\Delta_m \cup \Delta_{m-1})$ from the argument principle it follows that all the accumulation points of zeros of $a_{n,j} - a_{n,m}\widehat{s}_{m,j+1}$ must be contained in $\widetilde{\Delta}_{m,j} \cup \widetilde{\Delta}_{m-1,j}$. Consequently,  \eqref{City} with $h_{n,j}\equiv 1, n \geq 0,$ holds true uniformly on each compact subset of $\mathbb{C} \setminus (\widetilde{\Delta}_{m,j }\cup \widetilde{\Delta}_{m-1,j})$.
\end{remark}

\begin{remark} \label{rem2} Regarding the forms $\mathcal{B}_{n,j}$ and the polynomials $(P_{n,j}), j=0,\ldots,m, n\in \mathbb{Z}_+$ introduced in Subsection 2.3 in connection with the reversed Nikishin system $\mathcal{N}(\sigma_m,\ldots,\sigma_1)$,  asymptotic formulas analogous to those presented in this Section immediately follow and their formulation is left to the reader. We underline that the corresponding interaction matrix will be exactly the same that we had before and, therefore, we have the same equilibrium problem with the intervals taken in reversed order. An immediate consequence is
the following result.
\end{remark}

\begin{corollary} \label{weakPs}
Under the same hypothesis as in Theorem \eqref{teo4}, we have
\[
*\lim_n \mu_{P_{n,j}} =  {\lambda}_{m -j +1}, \qquad j=1,\ldots,m.
\]
where $\vec{\lambda} = (\lambda_1,\ldots,\lambda_m) \in
\mathcal{M}_1$ is the vector equilibrium measure determined by the
matrix $\mathcal{C}_{\mathcal{N}}$ on the system of compact
sets $\supp{\sigma_j}, j=1,\ldots,m$.
\end{corollary}

\section{Ratio asymptotic}

The ratio asymptotic of multiple orthogonal polynomials is described in terms of the branches of a conformal mapping defined on a Riemann surface associated with the geometry of the problem.

\subsection{The Riemann surface}
Consider the $(m+1)$-sheeted Riemann surface
$$
\mathcal R=\overline{\bigcup_{k=0}^m \mathcal R_k} ,
$$
formed by the consecutively ``glued'' sheets
$$
\mathcal R_0:=\overline {\mathbb{C}} \setminus
\ {\Delta}_1,\quad \mathcal R_k:=\overline {\mathbb{C}}
\setminus ( {\Delta}_k \cup
 {\Delta}_{k+1}),\,\, k=1,\dots,m-1,\quad \mathcal
R_m=\overline {\mathbb{C}} \setminus  {\Delta}_m,
$$
where the upper and lower banks of the slits of two neighboring
sheets are identified. (We remark that the sheets are made up of distinct points.)

\medskip

Let
$\psi $ be a conformal representation of $\mathcal{R}$
onto the extended complex plane satisfying
\[ \psi (z)=\frac{C_{1}}{z}+\mathcal{O}(\frac{1}{z^{2}}),\quad z\rightarrow\infty ^{(0)} \in \mathcal R_0\]
\[ \psi (z)=C_{2}\,z+\mathcal{O}(1),\quad z\rightarrow\infty ^{(m)} \in \mathcal R_m\]
where $C_{1}$ and $C_{2}$ are nonzero constants. Since the genus
of $\mathcal{R}$ is zero, $\psi $ exists and is uniquely
determined up to a multiplicative constant. Consider the branches
of $\psi $ corresponding to the different sheets
$k=0,1,\ldots,m$ of $\mathcal{R}$
\[\psi :=\{\psi_{k}\}_{k=0}^{m}\,. \]
We normalize $\psi $ so that
\begin{equation} \label{mormaliz}
\prod_{k=0}^{m}\,|\psi_{k}(\infty)|=1, \qquad C_1 >0.
\end{equation}

\medskip

Since $\psi $ is such that $C_1 >0,$
then
\[ \psi (z) = \overline{\psi (\overline{z})}, \qquad z \in
\mathcal{R}.
\]
In fact, define $\phi(z) := \overline{\psi (\overline{z})}$.
Notice that $\phi$ and $\psi $ have the same divisor (same poles and zeros counting multiplicities); consequently, there
exists a constant $C$ such that $\phi= C\psi $. Comparing the
leading coefficients of the Laurent expansion of these two functions
at $\infty^{(0)}$,  we conclude that $C=1$.

\medskip

In terms of the branches of $\psi $, the symmetry formula
above means that  for each  $k= 0,1,\ldots,m$:
\[
\psi_k: \overline{\mathbb{R}} \setminus
( {\Delta}_k\cup  {\Delta}_{k+1})
\longrightarrow \overline{\mathbb{R}}
\]
$( {\Delta}_0 =  {\Delta}_{m+1}=\emptyset)$;
therefore, the coefficients (in particular, the leading one) of the
Laurent expansion at $\infty$ of the branches are real numbers, and
\begin{equation}  \label{contacto}\psi_k(x_{\pm}) = \overline{\psi _k(x_{\mp})} =
\overline{\psi_{k+1}(x_{\pm})}, \qquad x \in
{\Delta}_{k+1}. \end{equation}

\medskip

Since $\lim_{x \to \infty} x \psi_0(x) = C_1 > 0$, by continuity it is not hard to deduce that $\psi_k(\infty) > 0, k=1,\ldots,m-1, $ and
$\lim_{x\to \infty} \psi_m(x)/x = \psi_m'(\infty) > 0.$ On the other hand, the product of all the branches $\prod_{k=0}^{m}
\psi_{k}$ is a single valued analytic function on
$\overline{\mathbb{C}}$ without singularities; therefore, by Liouville's
Theorem it is constant. Due to the previous remark and the normalization adopted in \eqref{mormaliz}, we can assert that
\begin{equation}\label{ident} \prod_{k=0}^{m}\,\psi_{k}(z) \equiv 1, \qquad z \in \overline{\mathbb{C}}
\end{equation}

\medskip

In \cite[Lemma 4.2]{AGR}  the following result was proved.

\begin{lemma} \label{unico} The system of boundary value problems
\begin{equation}          \label{eq:syst}
\{F_k \}_{k=1}^m:\quad \aligned &1 )\quad
F_k ,1/F_k  \in H(\mathbb{C} \setminus \Delta_k)\,,
\\
&2 )\quad F _k'(\infty) > 0, \quad k=1,\dots,m\,,
\\
&3 )\quad |F_k (x)|^2
\frac1{\bigl|(F_{k-1} F_{k+1} )(x)\bigr|}= 1
 \,, \quad x \in \Delta_k\,,
\endaligned
\end{equation}
(${F_0} \equiv {F_{m+1}} \equiv 1$) has a unique solution. The solution may be expressed by the formulas
\begin{equation} \label{eq:Fkl}
F_k  :=   \prod_{\nu=k}^m \psi_{\nu}, \qquad k=1,\ldots,m.
\end{equation}
\end{lemma}

That the functions defined by the product in \eqref{eq:Fkl} verify $1)$ is trivial. From the definition of $\psi$ it is also obvious that $F_k, k=1,\ldots,m$ has a simple pole at $\infty$. Since $\psi_k(\infty) >0, k=1,\ldots,m-1$ and $\psi_m'(\infty) >0$, we also have $2)$. That the boundary conditions $3)$ are satisfied  follows from  \eqref{contacto} and \eqref{ident}. The proof of unicity is more involved and can be checked in \cite[Lemma 4.2]{AGR}.

\subsection{Ratio asymptotic.} We will prove ratio asymptotic for all the polynomials $Q_{n,k}, k=1,\ldots,m,$ at once. Of course, the same can be obtained for the polynomials $P_{n,k}, k=1,\ldots,m$.  The precise statement is the following.

\begin{theorem} \label{teo5} Assume that $\sigma_k' > 0$ a.e. on  $\Delta_k, k=1,\ldots,m$. Then,
\begin{equation}\label{eq30}
\lim_{n\to \infty}\,\frac{Q_{n+1,k}(z)}{Q_{n,k}(z)}= \frac{{F} _{k}(z)}{F_k'(\infty)}, \qquad k=1,\ldots,m,
\end{equation}
uniformly on each compact subset of $\mathbb{C}\setminus \Delta_k$, where $ {F}_{k}$ is defined in \eqref{eq:Fkl}.
\end{theorem}

\begin{proof}
For the proof   we proceed as follows. From Lemma \ref{l3}, for each $k=1,\ldots,m$ the family of functions $(Q_{n+1,k}/Q_{n,k}), n \in \mathbb{Z}_+,$ is uniformly bounded on each compact subset of $\mathbb{C} \setminus \Delta_k$. To prove \eqref{eq30} it suffices to show that for any $\Lambda \subset \mathbb{Z}_+$ such that
\begin{equation}\label{eq30*}
\lim_{n \in \Lambda}\,\frac{Q_{n+1,k}(z)}{Q_{n,k}(z)}= G_{k}(z), \qquad k=1,\ldots,m,
\end{equation}
exists, the limiting functions $G_k$ do not depend on $\Lambda$. To achieve this, we will prove that there are positive constants $c_1,\ldots,c_m$ for which the collection of functions  $\{c_k G_k\}_{k=1}^m$ verifies the system \eqref{eq:syst}. Once this is done, using Lemma \ref{unico} and the fact that $G_k'(\infty) = 1, k=1,\ldots,m$, one obtains that
\begin{equation} \label{ratio}
c_kG_k = F_k, \qquad c_k = F_k'(\infty), \qquad k=1,\ldots,m,
\end{equation}
and \eqref{eq30} follows.

\medskip

Obviously, the functions in $\{G_k \}_{k=1}^m$ satisfy $1)$ and as mentioned before $G_k'(\infty) =1, k=1,\ldots,m,$ so $2)$ also takes place. We show that boundary conditions of type $3)$ are also valid with different values on the right hand side. In order to prove this, we use results on ratio and relative asymptotic of polynomials orthogonal with respect to varying measures developed in \cite{kn:B-G}, \cite{kn:Gui1}, \cite{kn:Gui2}.

\medskip

Along with the constants $K_{n,k}, k=1,\ldots,m$ defined in \eqref{Knj} we also consider the constants
\[ K_{n,m+1} = 1 \;, \quad \kappa_{n,k} = \frac{K_{n,k}}{K_{n,k+1}}
\;, \quad k=1,\ldots,m \;.
\]
Define
\begin{equation} \label{eq:orton} q_{n,k} =
\kappa_{n,k}Q_{n,k} \;, \qquad {h}_{n,k} =
K_{n,k+1}^2 \mathcal{H}_{n,k} \;, \qquad k = 1,\ldots,m \;.
\end{equation}
where $\mathcal{H}_{n,k}$ was defined in \eqref{Hnj}. From \eqref{int3} it follows that for each $k=1,\ldots,m$
\begin{equation} \label{int6}
\int x^{\nu}  Q_{n,k}(x)  \frac{|h_{n,k}(x)|{\rm d} \sigma_{k}(x)}{|Q_{n,k-1}(x)Q_{n,k+1}(x)|} = 0, \qquad \nu = 0,\ldots, n-1.
\end{equation}
Therefore, $Q_{n,k}$ is the $n$-th monic orthogonal polynomial with respect to the varying measure
\[ {\rm d} \rho_{n,k}(x) := \frac{|h_{n,k}(x)|{\rm d}\sigma_{k}(x)}{|Q_{n,k-1}(x)Q_{n,k+1}(x)|}
\]
and $q_{n,k}$ is the $n$-th orthonormal polynomial with respect to the same varying measure.

\medskip

Reasoning as before, we obtain that $Q_{n+1,k}$ and $q_{n+1,k}$ are the monic orthogonal and the orthonormal polynomials,
respectively, with respect to the varying measure
\begin{equation} \label{eq:var2}
\frac{|h_{n+1,k}(x)|{\rm d}\sigma_k(x)}{|Q_{n+1,k-1}(x)Q_{n+1,k+1}(x)|} = \frac{|h_{n+1,k}(x)|}{|h_{n,k}(x)|} \frac{|Q_{n,k-1}(x)Q_{n,k+1}(x)|}{|Q_{n+1,k-1}(x)Q_{n+1,k+1}(x)|}
{\rm d}\rho_{n,k}(x) \,.
\end{equation}

On account of \eqref{eq30*}
\begin{equation} \label{eq:*}
\lim_{n \in \Lambda } \frac{|Q_{n,k-1}(x)Q_{n,k+1}(x)|}{|Q_{n+1,k-1}(x)Q_{n+1,k+1}(x)|} =
\frac{1}{|G_{k-1}(x)G_{k+1}(x)|},
\qquad k = 1,\ldots,m\,.
\end{equation}
uniformly on $\Delta_k$ ($G_{0}  =
 G_{m+1} = 1$). On the other hand,  from \eqref{int4} it follows that
 \begin{equation} \label{int4*}  {h}_{n,k}(z) = \int \frac{ q^2_{n,k+1}(x)}{z-x} \frac{ {h}_{n,k+1}(x){\rm d}\sigma_{k+1}(x)}{Q_{n,k}(x)Q_{n,k+2}(x)}.
\end{equation}
and using Theorem 9 of \cite{kn:B-G} we get
\begin{equation} \label{debil} \lim_{n \to \infty} |h_{n, k}(z)| = \frac{1}{|\sqrt{(z-b_{k+1})(z -
a_{k+1})}|},
\qquad k=0,\ldots,m-1,
\end{equation}
uniformly on compact subsets of $\mathbb{C} \setminus \Delta_{k+1} $,
where $\Delta_{k+1} = [a_{k+1},b_{k+1}]$ (notice that from the definition we
have  $h_{n,m} \equiv h_{n+1,m} \equiv 1).$ Therefore,
\[ \lim_{n \to \infty} \frac{|h_{n+1,k}(x)|}{|h_{n,k}(x)|} = 1\,, \qquad k=1,\ldots,m\,,
\]
uniformly on $\Delta_k$ which combined with \eqref{eq:*} gives
\begin{equation}     \label{eq:HH}
\lim_{n\in \Lambda }  \frac{|h_{n+1,k}(x)|}{|h_{n,k}(x)|} \frac{|Q_{n,k-1}(x)Q_{n,k+1}(x)|}{|Q_{n+1,k-1}(x)Q_{n+1,k+1}(x)|} =
\frac{1}{|G_{k-1}(x)G_{k+1}(x)|},
\end{equation}
uniformly on $\Delta_k .$ The function on the right hand side of
this relation is continuous and different from zero on $\Delta_k.$

\medskip

Fix $k=1,\ldots,m.$ Let $Q_{n,k}^*$ be the $n$-th monic orthogonal polynomial with respect to the measure in \eqref{eq:var2}.
Write
\[ \frac{Q_{n+1,k}}{Q_{n,k}} = \frac{Q_{n+1,k}}{Q_{n,k}^*}\frac{Q_{n,k}^*}{Q_{n,k}}.
\]
Using the result on ratio asymptotic of orthogonal polynomials with respect to varying measures given in \cite[Theorem 6]{kn:B-G} it follows that
\begin{equation} \label{f1} \lim_{n \to \infty} \frac{Q_{n+1,k}(z)}{Q_{n,k}^*(z)} = \frac{\varphi_{k}(z)}{\varphi_{k}'(\infty)},
\end{equation}
uniformly on compact subsets of $\mathbb{C} \setminus \Delta_k$, where $\varphi_{k}$ denotes the conformal representation of $\overline{\mathbb{C}} \setminus \Delta_k$ onto $\{w: |w| > 1\}$ such that $\varphi_{k}(\infty) = \infty$ and $\varphi_{k}'(\infty) >0$. On the other hand, due to \eqref{eq:var2} and \eqref{eq:HH}, the result on relative asymptotic of orthogonal polynomials with respect to varying measures contained in \cite[Theorem 2]{DBG}   establishes that
\begin{equation} \label{f2} \lim_{n \in \Lambda} \frac{Q_{n,k}^*(z)}{Q_{n,k}(z)} = \frac{S_{k}(z)}{S_{k}(\infty)},
\end{equation}
where $S_k$ is the Szeg\H{o} function on $\overline{\mathbb{C}}
\setminus \Delta_k$ with respect to the weight
\[
{|G_{k-1}(x)G_{k+1}(x)|^{-1}}, \qquad
x \in \Delta_k;\]
therefore.
\begin{equation} \label{f3} |S_k(x)|^2 |G_{k-1}(x)G_{k+1}(x)|^{-1} = 1, \qquad x \in \Delta_k.
\end{equation}

\medskip

From \eqref{eq30*}, \eqref{f1}, and \eqref{f2}, it follows that
\begin{equation} \label{f5} \lim_{n\in \Lambda}\,\frac{Q_{n+1,k}(z)}{Q_{n,k}(z)}= G_k(z) = \frac{S_k(z) \varphi_k(z)}{S_k(\infty)\varphi_k'(\infty)}, \qquad k=1,\ldots,m.
\end{equation}
which combined with \eqref{f3} implies that for $x \in \Delta_k$, and $k=1,\ldots,m$
\begin{equation} \label{f4}
\frac{|G_k(x)|^2}{|G_{k-1}(x)G_{k+1}(x)|} = \frac{|S_k(x) \varphi_k(x)|^2}{(S_k(\infty)\varphi_k'(\infty))^2|G_{k-1}(x)G_{k+1}(x)|} = \frac{1}{w_k}
\end{equation}
where $w_k = (S_k(\infty)\varphi_k'(\infty))^2$
Therefore, the collection of functions $\{G_k\}_{k=1}^m$ fulfills \eqref{eq:syst} with right hand side in $3)$ equal to $1/w_k$.

\medskip

In order to get the correct value on the right hand side we need to find positive constants $c_k, k=1,\ldots,m$ such that
\[ \frac{c_k^2}{w_kc_{k-1}c_{k+1}} =1, \qquad c_0 = c_{m+1}= 1.
\]
Taking logarithms, it is sufficient to notice that the system of equations
\begin{equation}\label{norm-equa} 2\log c_k - \log c_{k-1} - \log c_{k+1} = \log w_k, \qquad k=1,\ldots,m
\end{equation}
has a solution since the determinant of this system is different from zero.
\end{proof}

The following Corollary complements Theorem \ref{teo5}.

\begin{corollary} \label{cor3}
Assume that $\sigma_k' >0$ a.e. on  $\Delta_k, k=1,\ldots,m$. Let $\{q_{n,k} = \kappa_{n,k}Q_{n,k} \}_{k=1}^m ,n\in \mathbb{Z}_+,$ be the
system of orthonormal polynomials defined in $(\ref{eq:orton})$
and $\{K_{n,k}\}_{k=1}^m ,n\in\mathbb{Z}_+,$ the values
given in $(\ref{Knj})$. Then, for each fixed $k = 1,\ldots,m,$ we
have
\begin{equation} \label{eq:xa}\lim_{n\in
{\Lambda}}\frac{\kappa_{n+1,k}}{\kappa_{n,k}}=
\kappa_k\,,
\end{equation}
\begin{equation} \label{eq:xk}
\lim_{n\in {\Lambda}}\frac{K_{n+1,k}}{K_{n,k}}= \kappa _1\cdots\kappa _k\,,
\end{equation}
and
\begin{equation} \label{eq:xb}
\lim_{n\in {\Lambda}}\frac{q_{n+1,k}(z)}{q_{n,k}(z)}= \kappa _k \frac{{F_k }(z)}{F_k'(\infty)},
\end{equation}
uniformly on compact subsets $\mathbb{C} \setminus \Delta_k$,  where
\begin{equation} \label{eq:xc} \kappa _k =
\frac{c_{k} }{\sqrt{c_{k-1}  c_{k+1} }}\,, \qquad
c_{k}  =
F_k^{\prime}(\infty), \qquad  k=1,\ldots,m.
\end{equation}
where $c_0 = c_{m+1} = 1$. We also have
\begin{equation} \label{a} \lim_{n\to \infty} \left|\frac{\mathcal{A}_{n+1,k}(z)}{\mathcal{A}_{n,k}(z)}\right| = \frac{1}{\kappa_1^2\cdots\kappa_{k+1}^2} \frac{F_{k+1}'(\infty)}{F_{k}'(\infty)}\left|\frac{F_{k}(z)}{F_{k+1}(z)}\right|, \qquad k=0,\ldots,m-1,
\end{equation}
uniformly on compact subsets of $\mathbb{C} \setminus (\Delta_k \cup \Delta_{k+1}).$ When $k=0$, $\Delta_0 = \emptyset$ and the factors $F_0$ and $F_0'(\infty)$ are substituted by $1$. Finally
\begin{equation} \label{b} \lim_{n\to \infty}  \frac{a_{n+1,k}(z)}{a_{n,k}(z)}  =  \frac{F_{m}(z)}{F_{m}'(\infty)}, \qquad k=0,\ldots,m-1,
\end{equation}
uniformly on compact subsets of $\mathbb{C} \setminus \Delta_m$.
\end{corollary}

{\bf Proof.} By  \eqref{eq30} we have limit in
(\ref{eq:HH}) as $n \to \infty$.  Reasoning as in
the deduction of formula \eqref{f5}
but now in connection with orthonormal polynomials (see
\cite{kn:B-G})  it follows that
\[
\lim_{n\to \infty} \frac{Q_{n+1,k}(z)}{Q_{n,k}(z)}=  \frac{(S_k\varphi_k)(z)}{(S_k\varphi'_k)(\infty)} \,, \qquad k = 1,\ldots,m,
\]
and
\begin{equation} \label{qnk}
\lim_{n\to \infty} \frac{q_{n+1,k}(z)}{q_{n,k}(z)}=  (S_k\varphi_k)(z) \,, \qquad k = 1,\ldots,m,
\end{equation}
uniformly on compact subsets of $\mathbb{C}\setminus \Delta_k$. Dividing the second of these limits by the first it follows that
\[ \lim_{n\to \infty}\frac{\kappa_{n+1,k}}{\kappa_{n,k}}= \kappa_k = \sqrt{w_k} =
\frac{c_k}{\sqrt{c_{k-1}c_{k+1}}}\,,
\]
where $w_k = (S_k(\infty)\varphi'(\infty))^2$ and the $c_k$ are
the normalizing constants found  solving the linear
system of equations \eqref{norm-equa}.  In \eqref{ratio} we saw that $c_k = F_k^{\prime}(\infty), k=1,\ldots,m$  and formula \eqref{eq:xa}  follows with $c_k$ as in \eqref{eq:xc}. Then, \eqref{eq:xb} is a consequence of \eqref{eq:xa} and \eqref{eq30}.

\medskip

From the definition of $\kappa_{n,k}\,,$ we have that
\[ K_{n,k} = \kappa_{n,1}\cdots\kappa_{n,k} \,.
\]
Taking the ratio of these constants for the multi-indices $n$ and $n+1$ and using (\ref{eq:xa}), we get (\ref{eq:xk}).

\medskip

Combining \eqref{Hnj}, \eqref{eq:orton}. and  \eqref{int4*} we obtain the formula
\[ \mathcal{A}_{n,k}(z) = \frac{1}{K_{n,k+1}^2}\frac{Q_{n,k}(z)}{Q_{n,k+1}(z)} \int \frac{q_{n,k+1}^2(x)}{z-x} \frac{h_{n,k+1}(x) \mbox{d}\sigma_{k+1}(x)}{Q_{n,k}(x)Q_{n,k+2}(x)}, \qquad k=0,\ldots,m-1.
\]
Taking the absolute value of the ratio of these expressions for $n$ and $n+1$, on account of \eqref{eq30}, \eqref{eq:xk}, and \eqref{debil} relation \eqref{a} immediately follows.

\medskip

According to \eqref{conver2}
\[ \lim_{n\to \infty} \frac{a_{n+1,k}(z)}{a_{n+1,m}(z)}\frac{a_{n,m}(z)}{a_{n,k}(z)} = \lim_{n\to \infty} \frac{a_{n+1,k}(z)}{a_{n,k}(z)}\frac{a_{n,m}(z)}{a_{n+1,m}(z)} = 1,\]
uniformly on compact subsets of $\mathbb{C} \setminus \Delta_m$. However, $a_{n,m} = (-1)^m Q_{n,m}$ and, therefore,
\[ \lim_{n \to \infty} \frac{a_{n+1,m}(z)}{a_{n,m}(z)} = \frac{F_m(z)}{F_{m}'(\infty)}
\]
so \eqref{b} takes place. We are done.  \hfill $\square$

\medskip

\begin{remark} From   $(\ref{eq:xa}),$ and $(\ref{eq:xk}),$
it follows that for each $k = 1,\ldots,m$
\[ \lim_{n \to \infty} \kappa_{n,k}^{1/n} =
\kappa_k, \qquad \lim_{n \to \infty} K_{n,k}^{1/n} = \kappa_1 \cdots \kappa_k,
\]
and
\[\lim_{n \to \infty} |Q_{n,k}(z)|^{1/n} =
| {F_k}(z)/F_k'(\infty)|,
\]
uniformly on compact subsets of ${\mathbb{C}}\setminus \Delta_k$. Compare with \eqref{conv-Qnj}.
\end{remark}

\begin{remark} \label{pes} Since the generating measures of $\mathcal{N}(\sigma_m,\ldots,\sigma_1)$ and $\mathcal{N}(\sigma_1,\ldots,\sigma_m)$ are the same but in reversed order, from Theorem \ref{teo5} and Corollary \ref{cor3} similar results can be formulated for the polynomials $(P_{n,k}), n \in \mathbb{Z}_+, k=1,\ldots,m,$ the corresponding orthonormal polynomials, their leading coefficients and the associated linear forms. The specific statements are left to the reader. We limit ourselves to the following statement.
\end{remark}

\begin{corollary} \label{pesratio}Under the assumptions of Theorem \ref{teo5}, we have
\begin{equation} \label{pesratio*} \lim_{n\to \infty}\,\frac{P_{n+1,k}(z)}{P_{n,k}(z)}=\frac{F_{m-k+1}(z)}{F_{m-k+1}'(\infty)}, \qquad k=1,\ldots,m,
\end{equation}
uniformly on compact subsets of $\mathbb{C} \setminus \Delta_{m-k+1}$.
\end{corollary}

{\bf Proof.} The existence of the limit follows directly from Theorem \ref{teo5}. To determine the expression of the limiting functions we need to construct the Riemann surface taking the intervals $\Delta_k$ in reversed order. But this is the same Riemann surface that we had before except that the sheets are in inverted order. Let $\widetilde{\psi}$ be the conformal map from $\mathcal{R}$ onto $\overline{\mathbb{C}}$ with a simple zero at $\infty^{(m)}$ and a simple pole at $\infty^{(0)}$. Let $ \widetilde{\psi}_k $ denote its branch on the sheet $\mathcal{R}_{m-k}$. We normalize $\widetilde{\psi}$ so that
\[
\prod_{k=0}^{m}\,|\widetilde{\psi}_{k}(\infty)|=1, \qquad \lim_{z\to \infty} z\widetilde{\psi}_m(z) > 0.
\]
This normalization is the equivalent of \eqref{mormaliz} for this situation.

\medskip

From the definition and the normalization it is easy to see that
\[ \widetilde{\psi} = 1/\psi, \qquad \widetilde{\psi}_k = 1/\psi_{m-k}, \qquad k=0,\ldots,m.
\]
According to Theorem \ref{teo5} the limit on the right hand of \eqref{pesratio*} should be $\widetilde{F}_k/\widetilde{F}_k'(\infty)$ where
\[ \widetilde{F}_k = \prod_{\nu=k}^{m} \widetilde{\psi}_\nu = \left(\prod_{\nu=k}^{m}  {\psi}_{m-\nu}\right)^{-1} = \prod_{\nu = m-k+1 }^m \psi_\nu = F_{m-k+1}
\]
which is what we needed to prove (recall that $\prod_{\nu=0}^{m}  {\psi}_{\nu}\equiv 1$). \hfill $\square$

\medskip

{\bf Acknowledgement:} The authors would like to thank Erwin Miña-Díaz for fruitful discussions related with the proof of Theorem \ref{velocidad}.

\bibliographystyle{amsplain}

\end{document}